\pdfoutput=1
\documentclass[reqno, 11pt, a4paper, oneside]{amsart}
\usepackage{amsmath}
\usepackage{amsfonts}
\usepackage{amssymb}
\usepackage[ansinew]{inputenc}
\usepackage{graphicx}
\usepackage{amsbsy}
\usepackage{fancyhdr}
\usepackage{yfonts}
\usepackage{mathabx}
\usepackage{enumerate}

\numberwithin{equation}{section}
\newtheorem{teo}{Theorem}[section]
\newtheorem{prop}[teo]{Proposition}
\newtheorem{lema}[teo]{Lemma}

\newtheorem{coro}[teo]{Corollary}

\theoremstyle{definition}
\newtheorem{defi}[teo]{Definition}
\newtheorem{notation}[teo]{Notation}
\newtheorem{ej}[teo]{Example}

\theoremstyle{remark}

\newtheorem*{ack}{Acknowledgements}
\sffamily
\title{Stability under scaling in the local phases of multiplicative functions}
\author{Miguel N. Walsh}
\address{Departamento de Matemática e IMAS-CONICET, Facultad de Ciencias Exactas y Naturales, Universidad de Buenos Aires, 1428 Buenos Aires, Argentina}
\email{mwalsh@dm.uba.ar}

\begin{document}

\def\F{\mathbb{F}}
\def\Fqn{\mathbb{F}_q^n}
\def\Fq{\mathbb{F}_q}
\def\Fp{\mathbb{F}_p}
\def\Di{\mathbb{D}}
\def\E{\mathbb{E}}
\def\Z{\mathbb{Z}}
\def\Q{\mathbb{Q}}
\def\C{\mathbb{C}}
\def\R{\mathbb{R}}
\def\N{\mathbb{N}}
\def\H{\mathbb{H}}
\def\P{\mathbb{P}}
\def\T{\mathbb{T}}
\def\modp{\, (\text{mod }p)}
\def\modN{\, (\text{mod }N)}
\def\modq{\, (\text{mod }q)}
\def\modone{\, (\text{mod }1)}
\def\ZN{\mathbb{Z}/N \mathbb{Z}}
\def\Zp{\mathbb{Z}/p \mathbb{Z}}
\def\Zan{a^{-n}\mathbb{Z}/ \mathbb{Z}}
\def\Zal{a^{-l} \Z / \Z}
\def\Pr{\text{Pr}}
\def\leftsize{\left| \left\{}
\def\rightsize{\right\} \right|}

\begin{abstract}
We introduce a strategy to tackle some known obstructions of current approaches to the Fourier uniformity conjecture. Assuming GRH, we then show the conjecture holds for intervals of length at least $(\log X)^{\psi(X)}$, with $\psi(X) \rightarrow \infty$ an arbitrarily slowly growing function of $X$. We expect the methods should adapt to nilsequences, thus also showing that the Generalised Riemann Hypothesis implies close to exponential growth in the sign patterns of the Liouville function.
\end{abstract}

\maketitle

\bigskip

\section{Introduction}

\subsection{Statement of results}

The Chowla and Sarnak conjectures are two central problems in analytic number theory with a large amount of literature around them. In \cite{T} it was shown that after performing a logarithmic averaging, both problems become equivalent to a 'local uniformity' estimate for the Liouville function $\lambda(n)$, which predicts that if $H, X$ are parameters with $H \rightarrow \infty$ as $X \rightarrow \infty$, then
\begin{equation}
\label{o25}
 \frac{1}{X} \int_X^{2X} \sup \left| \sum_{x \le n \le x+H} \lambda(n) F(n) \right| dx  = o(H), 
\end{equation}
where the supremum is taken among all nilsequences $F$ with a fixed degree of nilpotency. In fact, it was shown in \cite{MRTTZ} that it would suffice to establish (\ref{o25}) in the range $H \ge (\log X)^{\varepsilon}$ for every $\varepsilon > 0$.

In this article we shall focus on the Fourier uniformity conjecture, which deals with the particular case where the nilsequences $F$ are given by linear phases $n \mapsto e(\alpha n)$. While technically simpler, this case should encompass the conceptual difficulties of the full conjecture and indeed, we strongly expect the methods developed in this article to extend to the general case of arbitrary nilsequences. 

Following the work in \cite{MRT, MRTTZ, W1, W2}, the conjecture is currently known to hold in the range $H \ge \exp( C (\log X)^{1/2} (\log \log X)^{1/2})$ and this is a crucial barrier to break, since it is the natural limit of many of the known techniques on these problems as well as of other related results (e.g. \cite{HR}). Very roughly speaking, if we partition $[X,2X]$ into intervals of the form $[x,x+H]$, then by multiplicativity the behaviour of $\lambda$ on any one interval influences its behaviour on $O(H)$ other intervals. Thus, in order to have enough intervals related to each other as to affect the integral in (\ref{o25}) we need to carry an iterative procedure that involves $\sim \frac{\log X/H}{\log H}$ steps. But with current technology, such iterations create an error of a constant factor at each step. As a consequence, this error exceeds the right-hand side of (\ref{o25}) as soon as $H$ goes below the aforementioned range, making this procedure inviable. 

In this article we shall introduce some new ideas to tackle certain known obstructions of current approaches to the Fourier uniformity conjecture. Among these is a general structure theorem for the set of phases of a possible counterexample, which opens up the possibility of carrying out iterative arguments with extremely small losses at each step. This morally reduces the problem to two specific types of potential counterexamples, which we shall discuss later in this introduction. Since assuming GRH these can be easily ruled out in the range $H \ge (\log X)^{\psi(X)}$, with $\psi(X) \rightarrow \infty$ an arbitrary slowly growing function of $X$, we arrive in particular at the following result.

\begin{teo}
\label{1}
Assume GRH. Then, given $\eta > 0$ there exist $C, K>1$ such that, for every complex-valued multiplicative function $g$ with $|g| \le 1$ satisfying
\begin{equation}
\label{A}
 \int_X^{2X} \sup_{\alpha} \left| \sum_{x \le n \le x+H} g(n) e(\alpha n) \right| dx \ge \eta H X,
 \end{equation}
for some $H \ge (\log X)^K$, we have $\mathbb{D} (g; CX^2/H^2,C) \le C$.
\end{teo}

Here we are writing as usual $\mathbb{D}(g;T,Q)$ for the 'pretentious' distance introduced by Granville and Soundararajan \cite{GS}:
$$ \mathbb{D}(g;T,Q) = \inf \left( \sum_{p \le T} \frac{1 - \text{Re}(g(p)p^{it}\chi(p))}{p} \right)^{1/2} ,$$
with the infimum taken over all $|t| \le T$ and all Dirichlet characters of modulus at most $Q$. In particular, we see that (\ref{A}) cannot hold for the Möbius and Liouville functions.

As it was mentioned before, we expect the methods to extend to nilsequences. By \cite[Theorem 5.1]{MRTTZ}, this would in particular imply that under GRH the number of sign patterns of length $k$ of the Liouville function is at least $2^{k^{1/\psi(k)}}$ with $\psi(k)$ an arbitrarily slowly growing function of $k$. In other words, that the Generalised Riemann Hypothesis implies close to exponential growth in the sign patterns of the Liouville function. We refer to \cite{FH, MR, MRT3, MRTTZ, Mc, T0, TT, Te} for recent work on this problem.

Following the discussion in \cite{MRTTZ}, one can identify three main obstacles of current approaches to the local uniformity conjecture. The first one is the need of some form of 'mixing' or 'expansion', which even assuming GRH we only know how to guarantee in the range of Theorem \ref{1}. We do not address this is issue in the present article and it is indeed the reason for our hypotheses in that theorem.

The second one was already mentioned and is the inefficiency of iterative arguments when $H$ is small, which as previously discussed we will be able to overcome through some structural results we shall establish. Finally, the third obstruction is the 'lack of enough primes' in the range $H \le \log X$ which stops one from creating a sufficiently large modulus out of the primes smaller than $H$ and was considered a central limitation of previous approaches (again, see \cite{MRT, MRTTZ} for some discussion). As it turns out, the methods of the present article also point at a way of resolving this issue, as will be explained towards the end of this introduction.

In summary, the obstacle that remains in this approach is to obtain some suitable form of 'expansion'. In this regard, since our use of GRH is rather direct, we cannot discard that a more refined argument could be used to extend our conditional result to the range $H=(\log X)^{\varepsilon}$, possibly by taking advantage of some further averaging.

\subsection{General notation} We will write $X \lesssim Y$ or $X=O(Y)$ to mean that there is some absolute constant $C$ with $|X| \le C Y$ and we shall write $X \sim Y$ if both $X \lesssim Y$ and $Y \lesssim X$ hold. If the implicit constants depend on some additional parameters, a subscript will be used to indicate this. For a finite set $S$ we write $|S|$ for its cardinality. We abbreviate $e(x):=e^{2 \pi i x}$ and write $\| \cdot \|$ for the distance to $0$ in $\R / \Z$. 

\subsection{Overview of the methods}

We will now discuss in some length the main ideas and the general organisation of the paper. We will begin in Section 2 by setting up some preliminaries. Among these is a mixing lemma obtained as in \cite{MRT, MRTTZ}, but using GRH as input. Although the proof of this mixing lemma is the only place where the Riemann Hypothesis is used, the lemma itself will be used several times throughout the arguments to obtain some needed 'expansion' properties.

In Section 2 we will also invoke some estimates from previous work that reduce the Fourier uniformity conjecture to an inverse problem, the proof of which will be our goal. Let us briefly recall the general idea. If Theorem \ref{1} fails, then we can find $\gtrsim X/H$ disjoint intervals in $[X, 2X]$ of the form $[x,x+H]$ such that in each such interval $g(n)$ has a large correlation with $e(\alpha_x n)$ for some $\alpha_x \in \R / \Z$. This gives us what we call an $H$-'configuration': a collection of elements in $\R \times \R / \Z$ whose first coordinates are $H$-separated points. If we write $\mathcal{J}$ for the $H$-configuration given by these elements $(x, \alpha_x)$, then one can show using Elliott's inequality and the large sieve that there must exist some $P$ equal to a small power of $H$, such that if we write $\mathcal{P}$ for the set of primes in $[P,2P]$, there are $\gtrsim \frac{X}{H} |\mathcal{P}|^2$ quadruples $( (x,\alpha_x), (y,\alpha_y) , p , q ) \in \mathcal{J}^2 \times \mathcal{P}^2$ with $x/p$ close to $y/q$ and $p \alpha_x$ close to $q \alpha_y$. Let us write $(x,\alpha_x) \sim_{p,q} (y,\alpha_y)$ when this happens. The objective is then to show that given any $H$-configuration with this property, there must exists some real number $T$ with $|T| \lesssim X^2/H^2$ and some positive integer $q$ with $q=O(1)$, such that for $\gtrsim X/H$ of the elements of $(x,\alpha_x) \in \mathcal{J}$ there exists some integer $a_x$ such that $\alpha_x$ is close to $T/x + a_x/q$. Once this is accomplished, a simple Taylor expansion shows that $g(n)$ must correlate in the corresponding intervals with $n^{iT} \chi(n)$ for some Dirichlet character $\chi$ of modulus $q$ and by an application of the results of \cite{MR, MR2, MRT2} this yields Theorem \ref{1} (these ideas originate in \cite{MRT}, which the reader may consult for further discussion).

To obtain this inverse theorem our starting point will be the strategy introduced in \cite{W1}, which we also quickly summarise. Let $\mathcal{J}$ be the $H$-configuration just described, write $\mathcal{Z}$ for a maximal set of $PH$-separated points in $[XP, 4 XP]$ and for each $z \in \mathcal{Z}$ let $K^z$ consist of those pairs $((x,\alpha_x),p) \in \mathcal{J} \times \mathcal{P}$ such that $z/p$ is close to $x$. Using certain 'contagion' arguments, it is shown in \cite{W1} that it is possible to partition each $K^z$ into disjoint sets $K_i^z$, $i \ge 0$, such that $|K_0^z| \gtrsim |\mathcal{P}|$ for a positive proportion of the elements $z$ and such that for each $i \ge 0$ there exists some $\alpha_i^z \in \R / \Z$ with $p \alpha_i^z$ close to $\alpha_x$ for every $( (x,\alpha_x) ,p ) \in K_i^z$. It then follows from the pigeonhole principle and the triangle inequality that if we consider the $PH$-configuration $\mathcal{J}_1$ given by the elements $(z,\alpha_0^z)$, this configuration will also contain $\gtrsim \frac{X}{H} |\mathcal{P}|^2$ relations of the form $(z,\alpha_0^z) \sim_{p,q} (z',\alpha_0^{z'})$. This allows us to iterate the same contagion arguments until we reach a configuration $\mathcal{J}_k$ and an element $(z,\alpha) \in \mathcal{J}_k$ with $z \sim X P^k$ such that, for a positive proportion of elements $(x,\alpha_x) \in \mathcal{J}$, we can find primes $p_1,\ldots,p_k$ such that $z/\prod_{i=1}^k p_i$ is close to $x$ and $\alpha \prod_{i=1}^k p_i $ is close to $\alpha_x$. Once this global relation is established, the desired result follows easily from Vinogradov's lemma (we refer to \cite{W1} for more details).

Starting from a configuration that has $\ge c_0 \frac{X}{H} |\mathcal{P}|^2$ relations, the above procedure thus returns a new configuration having $\gtrsim \frac{X}{H} |\mathcal{P}|^2$ relations. The problem, of course, is the worsening of the implicit constant, with the methods of \cite{W1} returning something in the order of $c_0^2$. Since the argument needs to be iterated $\sim \frac{\log X}{\log H}$ times, this becomes useless as soon as $H$ is somewhat small. Therefore, what one would ideally like to have is a more efficient form of the above process. However, notice that not only losing a constant factor is problematic, but even losing only an $o(1)$ term at each step would fail to be enough. Indeed, to carry the above process $\sim \frac{\log X}{\log H}$ times for small values of $H$, we would need to ensure that $c_0$ is replaced at each step by a quantity that is at worst of the form $c_0 - O(\frac{1}{\log X})$. Although such a small loss may seem unrealistic at first, it turns out to be possible to accomplish.

The first ingredient in the proof is to exploit a certain rigidity present in the sets $K_i^z$ described above and this is carried out in Section 3. We call this phenomenon 'parallel rigidity': relations between only a few elements of $K_i^z$ and $K_j^{z'}$ are enough to force a relation between the corresponding frequencies $\alpha_i^z$, $\alpha_j^{z'}$ and therefore between all the elements in $K_i^z$ and $K_j^{z'}$. This is a variant of the contagion arguments from \cite{W1}. We will now describe three examples that will later help us explain some further arguments.

A first instance of this goes as follows. Suppose we are given $p_1, p_2 \in \mathcal{P}$ and elements $((x,\alpha_x),p_1), ((y,\alpha_y),p_2) \in K_i^z$ and $((x',\alpha_{x'}),p_1), ((y',\alpha_{y'}),p_2) \in K_j^{z'}$. Then, if we can find primes $q_1, q_2 \in \mathcal{P}$ such that $(x,\alpha_x) \sim_{q_1,q_2} (x',\alpha_{x'})$ and $(y,\alpha_y) \sim_{q_1,q_2} (y',\alpha_{y'})$, this automatically implies that $q_1 \alpha_i^z$ must be close to $q_2 \alpha_j^{z'}$ and therefore $(z,\alpha_i^z) \sim_{q_1,q_2} (z',\alpha_j^{z'})$. This, in turn, implies that a corresponding relation holds between all elements of $K_i^z$ and $K_j^z$.

On the other hand, suppose we have elements $((x,\alpha_x),p_1), ((y,\alpha_y),p_2) \in K^z$ which we do not know to be related under $\sim$. In particular, we now do not assume they belong to the same $K_i^z$ in the partition of $K^z$. If we are given another element $z'$ and a pair $((x',\alpha_{x'}),p_1), ((y',\alpha_{y'}),p_2) \in K_i^{z'}$ with $(x,\alpha_x) \sim (x',\alpha_{x'})$ and $(y,\alpha_y) \sim(y',\alpha_{y'})$, this unfortunately does not ensure that $(x,\alpha_x) \sim (y,\alpha_y)$. However, a second manifestation of what we call parallel rigidity says that something close to this is true, in that if we have a third element $z''$ and a corresponding pair of elements $((x'',\alpha_{x''}),p_1), ((y'',\alpha_{y''}),p_2) \in K_j^{z''}$ which also satisfy $(x,\alpha_x) \sim (x'',\alpha_{x''})$ and $(y,\alpha_y) \sim (y'',\alpha_{y''})$, then now this is indeed enough to guarantee that $(x,\alpha_x) \sim (y,\alpha_y)$.

Finally, let us mention one last instance of this rigidity, which is that only very few of the relations $(x,\alpha_x) \sim (y,\alpha_y)$ between elements of $K^z$ can happen between elements $(x,\alpha_x) \in K_i^z$ and $(y,\alpha_y) \in K_j^z$ with $i \neq j$.

Even equipped with this rigidity, it is not easy to see how to proceed. One would be tempted to build a configuration assigning to each $z \in \mathcal{Z}$ a 'popular' value of $\alpha_i^z$ that depends on the values of $\alpha_j^{z'}$ that arise from the relations of the elements of $K^z$, but carrying this out seems to lead to the same iterative inefficiencies that are present in previous approaches. Instead, what we will do is to try to 'lift' the exact structure of relations present in $\mathcal{J}$ through a fixed prime $p^{\ast} \in \mathcal{P}$. By this we mean in particular that to each $(x,\alpha_x)$ we will associate an element in the new configuration that will have the form $(z, \alpha_x^{\uparrow})$ for some $z$ close to $p^{\ast} x$ (actually, we will adjust things so that $z$ is exactly $p^{\ast} x$) and where $\alpha_x^{\uparrow}$ is a pre-image of $\alpha_x$ under $p^{\ast}$. The fact that we can build this new configuration lifting all the elements by the same fixed prime $p^{\ast}$ turns out to be a dramatic advantage, both for the current approach as well as for other possible strategies.

Carrying out such a 'fixed-prime lifting' is the second main ingredient in the proof and will be done in Section 4. To understand the strategy it is a good idea look at the above rigidity properties probabilistically, with the lifting prime $p^{\ast}$ as the variable. Let then $(x,\alpha_x) \sim_{p,q} (y,\alpha_y)$ be a pair of elements in $\mathcal{J}$ related by primes $p,q \in \mathcal{P}$ and let $i,j,z,z'$ be such that $( (x,\alpha_x) , p^{\ast}) \in K_i^z$ and $( (y,\alpha_y), p^{\ast}) \in K_j^{z'}$. Our goal would be to show that a similar relation $(z, \alpha_x^{\uparrow}) \sim_{p,q} (z', \alpha_y^{\uparrow})$ holds in our new configuration, where $\alpha_x^{\uparrow}$ is the preimage of $\alpha_x$ by $p^{\ast}$ that lies closest to $\alpha_i^z$ and $\alpha_y^{\uparrow}$ is the preimage of $\alpha_y$ by $p^{\ast}$ that lies closest to $\alpha_j^{z'}$.

Let us for the time being assume that $K_i^z$ has size comparable to $|\mathcal{P}|$ and following our discussion about acceptable error terms, let us say an event happens almost surely if the probability it fails is $O( \frac{1}{\log X})$. Using that $|\mathcal{P}| \gtrsim (\log X)^K$ for some large value of $K$ and that by assumption, a positive proportion of the quadruples $( (x,\alpha_x), (y, \alpha_y),p,q) \in \mathcal{J}^2 \times \mathcal{P}^2$ with $x/p$ close to $y/p$ satisfy $(x,\alpha_x) \sim_{p,q} (y,\alpha_y)$, one can prove that in the situation of the previous paragraph we almost surely have many other elements $(x', \alpha_{x'}) \in K_i^z$, $(y', \alpha_{y'}) \in K^{z'}$ with $(x',\alpha_{x'}) \sim_{p,q} (y', \alpha_{y'})$. In each case, we know that $(x,\alpha_x) \sim (x',\alpha_{x'})$ since both elements belong to $K_i^z$, but we do not necessarily know that $(y, \alpha_y) \sim (y', \alpha_{y'})$. However, by the second form of parallel rigidity we discussed, if they fail to relate this would make $(x,\alpha_x), (x',\alpha_{x'})$ the unique pair of related elements to which they both simultaneously relate. This is unlikely and indeed, it allows us to prove that almost surely we have $(y,\alpha_y) \sim (y', \alpha_{y'})$ for many of these elements. But by the third instance of rigidity, we know this is also unlikely to happen unless $(y',\alpha_{y'}) \in K_j^{z'}$, which means that almost surely we have $(y',\alpha_{y'}) \in K_j^{z'}$ for many of these elements as well. In particular, we have thus obtained two elements of $K_i^z$ relating with two elements of $K_j^{z'}$. By the first form of parallel rigidity we explained, this means in turn that $p \alpha_i^z$ must be close to $q \alpha_j^{z'}$ and from here we can conclude that (almost surely) we have $(z, \alpha_x^{\uparrow}) \sim_{p,q} (z', \alpha_y^{\uparrow})$, as desired.

Although this is an accurate representation of the basic idea, carrying this out in practice is unfortunately a bit cumbersome. One of the reasons for this is that every time we apply 'parallel rigidity' estimates, the implicit constants in the relations worsen. This forces us to carefully study certain nested sequences of sets and in particular leads to the chain of parameters discussed in (\ref{27chain}). The detailed argument is implemented in Section 4.

Recall that for the above probabilistic argument we had to assume that $K_i^z$ has size comparable to $|\mathcal{P}|$. Because of this, the exact result we obtain deals with sets $\mathcal{J}_r$, $0 \le r \le s$, for some large parameter $s = O(1)$, with $\mathcal{J} = \bigcup_{r=0}^s \mathcal{J}_r$ and such that if $(x,\alpha_x) \in \mathcal{J}_r$ for some $0 \le r < s$ then its corresponding set $K_i^z$ has size $\gtrsim \frac{1}{(\log X)^{rB}} |\mathcal{P}|$ for some large constant $B$. Then, the conclusion of the above arguments is that if $(x, \alpha_x) \sim (y,\alpha_y)$ with $(x,\alpha_x) \in \mathcal{J}_r$ for some $0 \le r < s$, then almost surely we have that $(y,\alpha_y) \in \mathcal{J}_{r+1}$ and that the relation lifts to a corresponding relation $(z, \alpha_x^{\uparrow}) \sim (z',\alpha_y^{\uparrow})$ in the new configuration. This structural estimate forms the basis of the proof of Theorem \ref{1} and is entirely unconditional.

The third main ingredient of the article is developed in Section 5 and consists of what we may call a 'concentration-increment' argument, since it is much in the spirit of other increment or decrement arguments in the literature. This will allow us to exploit the above decomposition by studying the behaviour of a 'concentration' parameter $\mathcal{C}(\mathcal{A})$ associated to any configuration $\mathcal{A}$ and which essentially measures how many relations $(x,\alpha_x) \sim_{p,q} (y,\alpha_y)$ we have between elements of $\mathcal{A}$ in proportion to the size of $\mathcal{A}$. The crucial property of $\mathcal{C}(\mathcal{A})$ that will allow us to run an increment argument is that, by the mixing lemma, unless $\mathcal{A}$ is very small this quantity is always uniformly bounded from above by an absolute constant.

Let then $\mathcal{A}$ be a configuration with $\mathcal{C}(\mathcal{A}) \sim 1$, which is for instance the case if $\mathcal{A} = \mathcal{J}$. This hypothesis allows us to obtain sets $\mathcal{A}_0, \ldots, \mathcal{A}_s$ with the properties described above. The issue we have at present is that we have no information about how relations involving pairs of elements from $\mathcal{A} \setminus \bigcup_{j=0}^{s-1} \mathcal{A}_j$ lift to the new configuration. However, we will be able to show that if this set is large, then we can find some $0 \le t < s$ such that, writing $\mathcal{B}_1 = \bigcup_{j=0}^t \mathcal{A}_j$ and $\mathcal{B}_2 = \mathcal{A} \setminus \mathcal{B}_1$, there are relatively few relations between elements of $\mathcal{B}_1$ and  elements of $\mathcal{B}_2$. This, in turn, will imply that one of these sets must satisfy a rather clean concentration-inequality of the form $\mathcal{C}(\mathcal{B}_i) \ge \mathcal{C}(\mathcal{A}) \left( \frac{|\mathcal{B}_i|}{|\mathcal{A}|} \right)^{-1/2}$. This iterates nicely and since the concentration parameter is always bounded from above by an absolute constant, we eventually reach a subset $\tilde{\mathcal{A}}$ of $\mathcal{A}$ with $\mathcal{C}(\tilde{\mathcal{A}}) \ge \mathcal{C}(\mathcal{A}) \left( \frac{|\tilde{\mathcal{A}}|}{|\mathcal{A}|} \right)^{-1/2}$, and therefore $|\tilde{\mathcal{A}}| \gtrsim |\mathcal{A}|$, such that the corresponding set $\tilde{\mathcal{A}} \setminus \bigcup_{j=0}^{s-1} \tilde{\mathcal{A}}_j$ is very small. This finally allows us to conclude that there exists some prime $p^{\ast}$ and a corresponding 'lifted' configuration $\tilde{\mathcal{A}}^{\uparrow p^{\ast}}$, such that every relation in $\tilde{\mathcal{A}}$ lifts almost surely to a relation in $\tilde{\mathcal{A}}^{\uparrow p^{\ast}}$.

A nice feature of this concentration increment argument is that it enables us to relate configurations lying at different scales. The process goes as follows. We start with $\mathcal{J}$ and apply the above arguments to reach a subset $\mathcal{A}_{(0)} \subseteq \mathcal{J}$ with $\mathcal{C}(\mathcal{A}_{(0)}) \ge \mathcal{C}(\mathcal{J}) \left( \frac{|\mathcal{A}_{(0)}|}{|\mathcal{J}|} \right)^{-1/2}$ and a prime $p_1^{\ast}$ such that every relation between elements of $\mathcal{A}_{(0)}$ lifts almost surely to a relation between the corresponding elements in a lifted configuration $\mathcal{A}_{(0)}^{\uparrow p_1^{\ast}}$. This implies in particular that $\mathcal{C}(\mathcal{A}_{(0)}^{\uparrow p_1^{\ast}}) = \mathcal{C}(\mathcal{A}_{(0)}) + O ( \frac{1}{\log X})$. We then iterate this. Once we have a lifted configuration $\mathcal{A}_{(j-1)}^{\uparrow p_j^{\ast}}$, we apply the concentration increment argument to obtain a new configuration $\mathcal{A}_{(j)} \subseteq \mathcal{A}_{(j-1)}^{\uparrow p_j^{\ast}}$ satisfying by induction an estimate of the form 
$$\mathcal{C}(\mathcal{A}_{(j)}) \ge \mathcal{C}(\mathcal{A}_{(j-1)}^{\uparrow p_j^{\ast}}) \left( \frac{|\mathcal{A}_{(j)}|}{|\mathcal{A}_{(j-1)}^{\uparrow p_j^{\ast}}|} \right)^{-1/2} \ge \mathcal{C}(\mathcal{J}) \left( \frac{| \mathcal{A}_{(j)} |}{|\mathcal{J}|} \right)^{-1/2} + O \left(\frac{j}{\log X} \right).$$
Furthermore, we also obtain a new prime $p_{j+1}^{\ast}$ such that every relation in $\mathcal{A}_{(j)}$ lifts almost surely to a relation in a lifted configuration $\mathcal{A}_{(j)}^{p_{j+1}^{\ast}}$, which therefore satisfies $\mathcal{C} ( \mathcal{A}_{(j)}^{p_{j+1}^{\ast}}) = \mathcal{C} (\mathcal{A}_{(j)}) + O \left( \frac{1}{\log X} \right)$ and this allows us to iterate the above procedure.

The conclusion is that we can find a sequence of primes $p_1^{\ast}, \ldots, p_k^{\ast}$ and corresponding $H \prod_{t=1}^j p_t^{\ast}$-configurations $\mathcal{A}_{(j)}$ whose first coordinates lie in $[X \prod_{t=1}^j p_t^{\ast}, 2 X \prod_{t=1}^j p_t^{\ast}]$, such that $|\mathcal{A}_{(j)}| \gtrsim X/H$ and the number of quadruples $( (x,\alpha_x), (y,\alpha_y), p ,q) \in \mathcal{A}_{(j)}^2 \times \mathcal{P}^2$ with $(x,\alpha_x) \sim_{p,q} (y,\alpha_y)$ is $\gtrsim \frac{X}{H} |\mathcal{P}|^2$, where in both cases the implicit constants are uniformly bounded from below thanks to the above concentration inequalities.

From here we wish to conclude the proof as in \cite{W1}. With this in mind, we develop in Section 6 some variants of estimates obtained in \cite{W2}, which we will employ in Section 7 to show that if $k$ is sufficiently large and $(x_0, \alpha_{x_0}) \in \mathcal{A}_{(k)}$, then there are $\gtrsim X/H$ elements $(x,\alpha_x) \in \mathcal{J}$ for which we can find primes $p_1, \ldots, p_k$ with $x_0/\prod_{j=1}^k p_i$ close to $x$ and $\alpha_{x_0} \prod_{j=1}^k p_j$ close to $\alpha_x$. This is achieved through the mixing lemma and is a good place in this introduction to provide an example showing why the expansion properties provided by that lemma are indeed necessary.

\begin{ej}
\label{closed}
Let $S$ be a set of $H$-separated points in $[X,2X]$ of size $\gtrsim X/H$. Suppose we can find a partition $S_1, \ldots, S_m$ of $S$ into sets of comparable size, such that for every $1 \le i \le m$, there are $\gtrsim |S_i| |\mathcal{P}|^2$ quadruples $(x,y,p,q) \in S_i^2 \times \mathcal{P}^2$ with $x/p$ close to $y/q$. Then, given any choice of real numbers $T_1, \ldots, T_m \lesssim X^2/H^2$, if we consider the configuration that assigns to each element $x \in S_i$ the frequency $\alpha_x = T_i/x \, (\text{mod }1)$, we obtain a configuration that satisfies the hypotheses of the inverse conjecture. However, if $m$ is sufficiently large, it is easy to see that for most choices of the elements $T_i$ it does not satisfy its conclusion.
\end{ej}

This example illustrates rather neatly why we are unable to make our results unconditional. Assuming the Riemann Hypothesis and as long as $P$ is a sufficiently large power of $\log X$, we can use the mixing lemma to show that any set $S_i$ as in the example must have size at least $c X/H$ for some absolute constant $c > 0$, thus forbidding the parameter $m$ to be sufficiently large as to contradict the inverse conjecture. In fact, since $P$ needs to be of the form $H^c$ with $c$ sufficiently small with respect to $\eta$, this is what forces us to have $K$ depend on $\eta$ in Theorem \ref{1}. On the other hand, if such an example existed for a large value of $m$, we could not expect, for instance, an element $(x_0, \alpha_{x_0}) \in \mathcal{A}_{(k)}$ to be connected to $\gtrsim X/H$ elements of $\mathcal{J}$ in the way recently described. It seems reasonable to suspect that if one had an unconditional method for ruling out this specific example, this could allow us to make our arguments unconditional.

Let $(x_0, \alpha_{x_0}) \in \mathcal{A}_{(k)}$ be as before. It is easy to see that as long as $k-j$ is sufficiently large, we can find an element $(x,\alpha_x) \in \mathcal{A}_{(j)}$ and two different sets of primes $p_1^{(i)}, \ldots, p_{k-j}^{(i)}$, $i=1,2$, such that $x_0/\prod_{r=1}^{k-j} p_{r}^{(i)}$ is close to $x$ and $\alpha_{x_0} \prod_{r=1}^{k-j} p_r^{(i)}$ is close to $\alpha_x$. From this it follows that $\alpha_{x_0}$ is close to $\frac{a}{b}$ for some integers $a, b$ with $b$ dividing $\left| \prod_{r=1}^{k-j} p_r^{(1)} - \prod_{r=1}^{k-j} p_r^{(2)} \right|$. Unfortunately, when $H$ is small we do not have a good control on this difference and can therefore only deduce from this a bound of the form $b \lesssim X^{\varepsilon}$. This is problematic, since using the arguments of \cite{W1} this at best would allow us to conclude that for many elements $(x,\alpha_x) \in \mathcal{J}$ we have some integer $a_x$ with $\alpha_x = \frac{a_x}{b} + O(1/H)$ and this is of course meaningless if $b \gtrsim H$. 

The trick here is to first run the argument with $\mathcal{A}_{(j)}$ in place of $\mathcal{J}$ for a certain $0 < j < k$. The advantage is that the contagion arguments make the relations $(x,\alpha_x) \sim_{p,q} (y,\alpha_y)$ between elements of $\mathcal{A}_{(j)}$ become stronger as $j$ grows, in the sense that the distance between $p \alpha_x$ and $q \alpha_y$ gets smaller the larger $j$ is. In particular, if we know that $\alpha_x$ is close to $\frac{a_x}{b}$ and $\alpha_y$ is close to $\frac{a_y}{b}$, then for an adequate choice of $j$ we will be able to guarantee that $p a_x$ has to actually be \emph{equal} to $q a_y$ ($\text{mod }b$). This brings us to the second potential counterexample.

\begin{ej}
Let $b \in \N$ and let $S$ be a collection of $H$-separated points in $[X,2X]$. Suppose that to each $x \in S$ we can assign an integer $a_x$ coprime to $b$ in such a way that we have $\gtrsim \frac{X}{H} |\mathcal{P}|^2$ quadruples $(x,y,p,q) \in S^2 \times \mathcal{P}^2$ with $x/p$ close to $y/q$ and $p a_x \equiv q a_y \, (\text{mod }b)$. Then, if $b$ is sufficiently large, the configuration given by the elements $(x,\frac{a_x}{b})$ contradicts the inverse conjecture.
\end{ej}

This counterexample can be ruled out as long as we have some sort of 'modular' expansion, analogously as how the previous example was dealt with through what we may call 'archimedean' expansion. Indeed, assuming GRH we can easily show using the mixing lemma that for any such example it must be $b=O(1)$, which is acceptable. This is done in Section 7. On the other hand, this example may be an interesting model problem to consider, given that its 'exact' nature may allow for a different set of tools. It should be mentioned that the bound $b \lesssim X^{\varepsilon}$ we previously referred to can be improved a bit and indeed, the upper bounds one can obtain for $b$ largely depend on how efficient is the method we have to obtain 'archimedean' expansion. In particular, it is plausible that an adequate argument to deal with Example \ref{closed} could lead to a regime of $b$ that can be handled unconditionally, which is why we believe that first example is the most representative of the obstruction our methods face.

Going back to our arguments, we deduce that $\alpha_{x_0}$ is close to $\frac{a}{b}$ for some $a,b = O(1)$ and this allows us to write $\alpha_{x_0} = \frac{a}{b} + \frac{T}{x_0}$ for some $|T| \lesssim X^2/H^2$. Since we know that $(x_0,\alpha_{x_0})$ is connected through products of primes to $\gtrsim X/H$ elements of $\mathcal{J}$, we can then quickly conclude the proof of Theorem \ref{1}. This is actually carried out in two parts, obtaining the result with a weaker bound of the form $|T| \lesssim (X^2/H^2) \log X$ in Section 7 and then removing the extra logarithmic factor in Section 8.

One would of course like to extend Theorem \ref{1} to the range $H \ge (\log X)^{\varepsilon}$. Although the lack of expansion stops us from doing this, the methods we have explained do seem to rule out what was considered a major obstruction of previous approaches in this regime. Let us finish this introduction with a brief discussion of this issue.

The starting point of the methods of \cite{MRT, MRTTZ, W2} is the observation (made in \cite{MRT}) that for each element $(x,\alpha_x) \in \mathcal{J}$ we may use the Chinese Remainder Theorem to replace $\alpha_x$ with a congruent real number $\tilde{\alpha}_x$ such that for a positive proportion of the relations $(x,\alpha_x) \sim_{p,q} (y,\alpha_y)$, we have that $p \tilde{\alpha}_x$ and $q \tilde{\alpha}_y$ are close not just mod $1$, but to a large modulus $Q_{x,y}$. If these moduli are sufficiently large, in particular larger than the expected value of $|T|$, this opens up a number of tools that eventually allow one establish the conjecture in the range $H \ge \exp(C (\log X)^{1/2} (\log \log X)^{1/2})$. However, as it was noted in \cite{MRT}, since these moduli $Q_{x,y}$ are obtained as the product of distinct primes of size at most $H$, it is impossible to make them sufficiently large once $H$ drops below $\log X$.

On the other hand, the conditional methods of this article provide us with primes $p_1^{\ast}, \ldots, p_k^{\ast}$ and a configuration $\mathcal{A}_{(k)}$ in $[X \prod_{t=1}^k p_t^{\ast}, 2X \prod_{t=1}^k p_t^{\ast}]$ of size $\gtrsim X/H$ with $\gtrsim \frac{X}{H} |\mathcal{P}|^2$ relations of the form $(x,\alpha_x) \sim_{p,q} (y,\alpha_y)$ between its elements. Furthermore, for each $(x,\alpha_x) \in \mathcal{A}_{(k)}$, there exists some $(x',\alpha_{x'}) \in \mathcal{J}$ with $x/\prod_{t=1}^k p_t^{\ast} = x'$ and $\alpha_x \prod_{t=1}^k p_t^{\ast} = \alpha_{x'}$. If for each $\alpha_x$ we take a representative in $[0,1)$ and then redefine the corresponding $\alpha_{x'}$ as the real number $\alpha_x \prod_{t=1}^k p_t^{\ast}$, we obtain that every time we have a relation $(x,\alpha_x) \sim_{p,q} (y,\alpha_y)$ in $\mathcal{A}_{(k)}$, their corresponding elements in $\mathcal{J}$ satisfy both $(x',\alpha_{x'}) \sim_{p,q} (y',\alpha_{y'})$ and that $p \alpha_{x'}$ is close to $q \alpha_{y'}$ mod $Q = \prod_{t=1}^k p_t^{\ast}$. Since our methods allow for a large choice of $k$, this not only makes $Q$ sufficiently large as to avoid the aforementioned limitation but also makes all the choices of $Q_{x.y}$ equal to a fixed $Q$, which is a substantial advantage in practice. Of course, in order to profit from this it still remains to deal with the expansion issue. 

Finally, the alert reader may notice that we have taken advantage of the fact that $H$ is a large power of $\log X$ in other parts of the arguments, for example when carrying out the probabilistic approach. However, these uses seem less fundamental and there appear to be a number of ways to remedy them if necessary. For example, by choosing a large positive integer $k=O(1)$, iterating the contagion arguments $k$ times at the beginning of the proof and then having the role of $\mathcal{P}$ played by its $k$-fold product $\mathcal{P}^k$. 

\begin{ack}
This work was partially supported by a von Neumann Fellowship from the Institute for Advanced Study and the National Science Foundation under Grant No. DMS-1926686.
\end{ack}

\section{Preliminaries}

In this section we develop some preliminary estimates we shall need. In the first part, we recall some results from previous work that reduce Theorem \ref{1} to an inverse problem. We then use GRH to establish a variant of the mixing lemma from \cite{MRT} that we will be using throughout the article. Finally, we end this section with a brief discussion of the relations between the parameters that will be used in the proof of Theorem \ref{1}.

\subsection{Configurations}

Let us begin by formally defining our notion of configurations.

\begin{defi}
Given $\tilde{H} > 1$, $c > 0$ and an interval $I \subseteq \R$, we say $\mathcal{J} \subseteq I \times \R$ is a $(c,\tilde{H})$-configuration in $I$ if $|\mathcal{J}| \ge c|I|/\tilde{H}$ and the first coordinates are $\tilde{H}$-separated points in $I$ (i.e. $|x-y| \ge \tilde{H}$ if $x \neq y$). 
\end{defi}

The next lemma, which goes back to \cite{MRT}, shows that a counterexample to Theorem \ref{1} would lead to a specific type of configuration.

\begin{lema}
\label{base00}
Let $\varepsilon > 0$ be sufficiently small with respect to $\eta$ and assume Theorem \ref{1} fails for some choice of $K$ that is sufficiently large with respect to $\varepsilon, \eta$. Then, there exist $c_0, c_0', \delta_0, C_0 \sim_{\eta, \varepsilon} 1$ with $\delta_0 < 10^{-10}$, $P \in [H^{\varepsilon^2}, H^{\varepsilon}]$ and a $(c_0, H)$-configuration $\mathcal{J}$ in $[X,2X]$ such that, writing $\mathcal{P}$ for the set of primes in $[P,2P]$, we have that for at least $\ge c_0' \frac{X}{H} |\mathcal{P}|^2$ pairs $(x,\alpha_x), (y,\alpha_y) \in \mathcal{J}$, there exists $p,q \in \mathcal{P}$ with $|x/p - y/q| < \delta_0 H/P$ and $\| p \alpha_x - q \alpha_y \| \le C_0 /(PH)$.
\end{lema} 

\begin{proof}
This follows from \cite[Lemma 3.6]{W1}, upon applying the pigeonhole principle and the triangle inequality.
\end{proof}

\begin{notation}
For the rest of the article we will let $P, \mathcal{P}$ be as in Lemma \ref{base00}.
\end{notation}

As in previous work \cite{MRT, MRTTZ, W1, W2}, we will proceed by establishing an inverse theorem that implies Theorem \ref{1}. Precisely, using the arguments at the end of \cite[Section 5]{MRT} or at the end of \cite[Section 4]{W1}, one easily sees that Theorem \ref{1} follows from the following estimate, which will be our goal in the remainder of this article.

\begin{teo}
\label{2}
Let $\mathcal{J}$ be as in Lemma \ref{base00}. Assume GRH and that $K$ is sufficiently large with respect to $\eta$. Then, there exist $T \in \R$ with $|T| \lesssim_{\eta} X^2/H^2$ and $b \in \N$ with $b=O_{\eta}(1)$ such that, for $\gtrsim_{\eta} X/H$ elements $(x,\alpha_x) \in \mathcal{J}$, there exists $a_x \in \N$ with $\| \alpha_x - \frac{a_x}{b} - \frac{T}{x} \| \lesssim_{\eta} H^{-1}$.
\end{teo}

We now need a notion of 'paths' similar to that in \cite{W2}.

\begin{defi}
\label{path}
Given $P < \tilde{H} < Y$ and a set $A \subseteq [Y/10,10Y]$, we define a path of length $k$ inside $A$ to be a sequence of elements $x_1,\ldots,x_{k+1} \in A$ such that, for every $1 \le i \le k$, there exist primes $p_i, q_i \in \mathcal{P}$ with $|x_i/p_i - x_{i+1}/q_i| \lesssim \tilde{H}/P$. We call $x_1$ the initial point of the path, $x_{k+1}$ the end point and say the path connects $x_1$ and $x_{k+1}$.
\end{defi}

We then have the following estimate from \cite{W2}.

\begin{lema}
\label{pd21}
There exists $\epsilon \sim 1$ such that, for every path as in Definition \ref{path} of length $k \le \epsilon \log (Y/\tilde{H})$, we have
$$|x_1 \prod_{i=1}^{k} q_i/p_i - x_{k+1}|  \lesssim k \tilde{H}.$$
\end{lema}

\begin{proof}
This follows from \cite[Corollary 4.2]{W2}. Similar arguments will be performed explicitly in Section 6.
\end{proof}

\subsection{Mixing lemmas}
\label{mixing lemmas}

We will be using the following standard consequence of GRH and Perron's formula (see for example \cite[Section 4]{Tb}).

\begin{lema}
\label{GRHbound}
Let $\chi$ be a Dirichlet character of modulus $q$ and write $\Lambda$ for the von Mangoldt function. Assuming the Generalised Riemann Hypothesis, we have the bound
$$ \frac{1}{P} \sum_{n \le P} \Lambda(n) \chi(n) n^{it} \lesssim \frac{\delta(\chi)}{|t|+1} + P^{-c},$$
for every $P \gtrsim \log^C (q(2 + |t|))$ and some absolute constants $c,C > 0$, where $\delta(\chi)=1$ if $\chi$ is a principal character and $0$ otherwise.
\end{lema}

As we have discussed, we will rely on a version of the mixing lemma proved in \cite[Lemma 5.1]{MRT}, but which uses Lemma \ref{GRHbound} as input. We now proceed to establish this.

\begin{lema}
\label{ML}
Let $a,Q \in \N$ with $\gcd(a,Q)=1$ and let $P \le \tilde{H} \le \frac{Y}{\log^2 Y}$. Let $\mathcal{A}_1,\mathcal{A}_2$ be $\tilde{H}$-separated subsets of $[Y/10,10Y]$ and let $\log Q \le \log Y \le P^{\varepsilon}$ for some sufficiently small $\varepsilon > 0$. Let $\epsilon > 0$ be as in Lemma \ref{pd21}. Then, the number of paths in $\mathcal{A}_1 \cup \mathcal{A}_2$ of length $k \le \epsilon \log (Y/\tilde{H})$ connecting an element of $\mathcal{A}_1$ with an element of $\mathcal{A}_2$ and such that the corresponding primes $(p_1,\ldots,p_{k},q_1,\ldots,q_{k})$ are coprime to $Q$ and satisfy
\begin{equation}
\label{convmod}
  \prod_{i=1}^{k} p_i \equiv a \prod_{i=1}^{k} q_i \, \,  (\text{mod }Q),
  \end{equation}
is bounded by
$$ \lesssim \frac{C^{k}}{\varphi(Q)}  \frac{\tilde{H}}{Y}  |\mathcal{P}|^{2k} |\mathcal{A}_1| |\mathcal{A}_2| + C^{k} |\mathcal{A}_1|^{1/2} |\mathcal{A}_2|^{1/2} |\mathcal{P}|^{2k}  \varphi (Q) P^{-ck},$$
for some absolute constants $c,C > 0$.
\end{lema}

\begin{proof}
Recall that we write $\mathcal{P}$ for the set of primes in $[P,2P]$ and we shall also write $\mathcal{P}^k_u$ for the $k$-tuples $(p_1,\ldots,p_k) \in \mathcal{P}^k$ with $\prod_{i=1}^k p_i \equiv u$ (mod $Q$). Let $\psi:\R \rightarrow \R$ be a non-negative Schwartz function with Fourier support in $[-1,1]$ such that $\psi(t) \ge 1$ for every $|t| \le 1$. If $y_1  \in \mathcal{A}_1,y_2 \in \mathcal{A}_2$ are connected via a path of length $k$ involving primes $(p_1,q_1,\ldots,p_k,q_k)$, it follows from Lemma \ref{pd21} that
\begin{equation}
\label{mil9}
 \psi (\frac{Y}{C_1k\tilde{H}} (\log y_1 - \log y_2 + \sum_{i=1}^k \log q_i - \sum_{i=1}^k \log p_i ) ) \ge 1,
 \end{equation}
as long as $C_1$ is a sufficiently large constant. Also, since $\mathcal{A}_1$ and $\mathcal{A}_2$ are $\tilde{H}$-separated sets, it easy to see that once $y_1,y_2$ and the primes $p_1,q_1,\ldots,p_k,q_k$ are fixed, there are at most $O(1)^k$ possible paths in $\mathcal{A}_1 \cup \mathcal{A}_2$ connecting $y_1$ with $y_2$ and consisting of these primes. By (\ref{convmod}) and (\ref{mil9}), it will therefore suffice to bound the sum over all such pairs $y_1,y_2$ of 
$$  \sum_{uv^{-1} \equiv a \, (Q)} \sum_{(p_1 ,\ldots, p_k) \in \mathcal{P}_u^k} \sum_{(q_1 ,\ldots, q_k) \in \mathcal{P}_v^k} \psi \left(\frac{Y}{C_1k\tilde{H}} \left( \log y_1 - \log y_2 + \sum_{i=1}^k \log q_i - \sum_{i=1}^k \log p_i \right) \right).$$
By Fourier inversion and a change of variables, this is bounded by
\begin{equation}
\label{reg}
 \lesssim \frac{k\tilde{H}}{Y} \int_{|\xi| \le Y/(C_1 k\tilde{H})} |S_1(\xi)| |S_2(\xi)| |T(\xi)| d \xi,
 \end{equation}
where 
$$S_i(\xi) = \sum_{y \in \mathcal{A}_i} e \left(  \xi \log y \right),$$
for $i=1,2$, and 
$$ T(\xi) = \sum_{uv^{-1} \equiv a(\text{mod }Q)} \overline{ \left( \sum_{(p_1 ,\ldots, p_k) \in \mathcal{P}_u^k} e \left( \xi  \sum_{i=1}^k \log p_i  \right)  \right) } \left( \sum_{(q_1,  \ldots, q_k) \in \mathcal{P}_v^k} e \left( \xi   \sum_{i=1}^k \log q_i  \right) \right).$$
We are going to bound the factors inside the sum defining $T(\xi)$. Expanding into Dirichlet characters, we have
\begin{align*}
 \sum_{(p_1, \ldots , p_k) \in \mathcal{P}_u^k }  \prod_{i=1}^k p_i^{2 \pi i \xi} &=\sum_{\chi}  \frac{\overline{\chi(u)}}{\varphi(Q)} \left( \sum_{p \in \mathcal{P}} \chi(p) p^{2 \pi i \xi} \right)^k .
\end{align*}
By Lemma \ref{GRHbound} and our hypothesis on the sizes of $P, Q$ and $Y$, this leads to a bound of the form
$$ |T(\xi) | \lesssim C^{2k} |\mathcal{P}|^{2k}  \left( \frac{1}{\varphi(Q)}  \left( \frac{1}{1+|\xi|} \right)^{2k} + \varphi (Q) P^{-2ck} \right),$$
for some $C \sim 1$. We now handle both terms in this bound separately in (\ref{reg}). For the first term, we use the trivial bounds $|S_1(\xi)| \le |\mathcal{A}_1|, |S_2(\xi)| \le |\mathcal{A}_2|$ and integrate  $\left( \frac{1}{1+|\xi|} \right)^{2k}$ to obtain a bound of the form
$$ \lesssim \frac{k\tilde{H}}{Y} \frac{C^{2k}}{\varphi(Q)} |\mathcal{P}|^{2k} |\mathcal{A}_1| |\mathcal{A}_2|.$$
For the second term, we just take it out of the integral and use Cauchy-Schwarz and the large sieve (in the form of \cite[Lemma 2.3]{MRT}) for $S_1$ and $S_2$, thus obtaining the bound
$$ \lesssim C^{2k} k |\mathcal{A}_1|^{1/2} |\mathcal{A}_2|^{1/2} |\mathcal{P}|^{2k}  \varphi (Q) P^{-2ck}.$$
The result follows combining both bounds and adjusting $C$ if necessary.
\end{proof}

We single out a consequence of the $k=Q=1$ case, which follows from RH, since it is the estimate we will be using most of the time.

\begin{coro}
\label{ML0}
Let $P < \tilde{H} < \frac{Y}{\log^2 Y}$. Let $\mathcal{A}_1, \mathcal{A}_2$ be sets of $\tilde{H}$-separated points in $[Y/10,10Y]$ and assume $\log Y < P^{\varepsilon}$ for some sufficiently small $\varepsilon > 0$. Then, there exists some $c_0 \gtrsim 1$ such that the number of quadruples $(x,y,p,q) \in \mathcal{A}_1 \times \mathcal{A}_2 \times \mathcal{P}^2$ with $|x/p - y/q| \lesssim \tilde{H}/P$ is at most
$$ \lesssim \frac{\tilde{H}}{Y} |\mathcal{A}_1| |\mathcal{A}_2| |\mathcal{P}|^2 + |\mathcal{A}_1|^{1/2} |\mathcal{A}_2|^{1/2} P^{2-c_0}.$$
\end{coro}

\subsection{Relations between the parameters}

For convenience of the reader we now summarise some of the main parameters appearing throughout the proof of Theorem \ref{2} and the relations between them. Almost all these parameters will depend on the value of $\eta$ in Theorem \ref{2} and, consequently, we will allow all implicit constants in the asymptotic notation to depend on $\eta$ without explicit mention. The only significant exceptions will be given by the absolute constants we have obtained in Section \ref{mixing lemmas}, which in turn lead to the absolute constant appearing in the statement of Lemma \ref{C1} and the absolute constant $c^{\ast}$ in Lemma \ref{grow26}.

We will start with the $(c_0,H)$-configuration $\mathcal{J}$ in $[X,2X]$ provided by Lemma \ref{base00} and recursively construct $(c_1,\tilde{H})$-configurations in intervals $[Y,2Y]$. In each case, we will have a sequence of primes $p_1^{\ast},\ldots,p_j^{\ast}$ such that $\tilde{H} = H \prod_{t=1}^j p_t^{\ast}$ and $Y = X \prod_{t=1}^j p_t^{\ast}$, so in particular we will always have $Y/\tilde{H} = X/H$. The parameters $c_0, c_0', \delta_0, C_0$ corresponding to $\mathcal{J}$ will be replaced by new parameters $c_1, c_2, \delta, C_2$ in these additional configurations, but we will crucially ensure that throughout the $\sim \frac{\log X/{H}}{\log P}$ steps of this recursion, we retain uniform bounds for these parameters (which, as usual, depend on $\eta$).

To accomplish this we will be studying families of sets $K_i^z$ associated to each configuration and this will require introducing parameters
\begin{equation}
\label{27chain}
 C_2 < C_1 = C[0] < C[1] < \overline{C}[1] < \ldots < C[s] < \overline{C}[s],
 \end{equation}
with each constant assumed sufficiently large with respect to the previous one. We will eventually need the parameter $s$ to be sufficiently large with respect to $\eta$. Because of the uniform bounds for $c_1, c_2, \delta, C_2$, one could in fact take $s$ to be fixed throughout the process. Finally, $K$ in Theorem \ref{1} will be chosen sufficiently large with respect to all of the constants we have discussed.

\section{Parallel rigidity}

We now turn to the first main ingredient of the proof, which is the idea of 'parallel rigidity' discussed in the introduction. We begin by setting up some notation that will be used throughout the article.

\begin{notation}[$P, \mathcal{P}, \varepsilon, \eta, \mathcal{J}, K, \sim$]
\label{not28}
Throughout the rest of this article we will let $P, \mathcal{P}, \varepsilon, \eta, \mathcal{J}, K$ be as in Lemma \ref{base00}. We will assume $\varepsilon \gtrsim_{\eta} 1$ is sufficiently small with respect to $\eta$ and allow all implied constants to depend on $\eta$. Given a dyadic interval $[Y,2Y]$, elements $(x,\alpha_x), (y, \alpha_y) \in [Y,2Y] \times \R / \Z$ and primes $q_1, q_2 \in \mathcal{P}$, we will write $(x,\alpha_x) \sim_{q_1,q_2} (y,\alpha_y)$ to mean that $|x/q_1 - y/q_2|  \le C \tilde{H}/P$ and $\| q_1 \alpha_x - q_2 \alpha_y \| \le CP/\tilde{H}$ for some $C \sim 1$ and where $\tilde{H}$ will always be clear from context. We will sometimes write $(x,\alpha_x) \sim_{q_1,q_2}^C (y,\alpha_y)$ if we wish to clarify the value of $C$ and, on the other hand, we will simply write $(x,\alpha_x) \sim (y,\alpha_y)$ to mean that there exist $q_1, q_2 \in \mathcal{P}$ with $(x,\alpha_x) \sim_{q_1,q_2} (y,\alpha_y)$.
\end{notation}

Let $P < \tilde{H} < Y$, $c \sim 1$ and consider a $(c,\tilde{H})$-configuration $\mathcal{A}$ in $[Y,2Y]$. We will proceed to define some sets $K_i^z$ associated with $\mathcal{A}$. These sets will not really be needed until Section 4, but will help us motivate the lemmas that we will prove in this section.

We let $\mathcal{Z}$ be a maximal set of $\delta P \tilde{H}$-separated points in $[PY, 4PY]$, for some small $1 \lesssim \delta < 1$ to be specified later, and to each $z \in \mathcal{Z}$ we associate a set $K^z$ consisting of all pairs $((x,\alpha_x),p) \in \mathcal{A} \times \mathcal{P}$ with $| z/p - x | < \tilde{H}/100$. In particular, we see that every prime appears at most once as a second coordinate in $K^z$.

We now let $ C_1 \sim 1$ be a constant to be specified later and construct a partition of $K^z$ in the following way. If there is no $((x,\alpha_x),p) \in \mathcal{A} \times \mathcal{P}$ with $|px - z| < 4 \delta P\tilde{H}$, we let $K_0^z$ be the empty set. Otherwise, we let $\tilde{K}_0^z$ be a set of elements $((x,\alpha_x),p) \in K^z$ with $|px - z| < 4 \delta P\tilde{H}$ of maximal size for which we can find a single frequency $\alpha_0^z \in \R / \Z$ with $\| p \alpha_0^z - \alpha_x \| < C_1/\tilde{H}$ for each such $((x,\alpha_x),p)$. We then let $K_0^z$ be the set of all elements $((x,\alpha_x),p) \in K^z$ satisfying this relation with $\alpha_0^z$. Once $K_i^z$ has been constructed for every $i < j$, we then let $K_j^z$ be a subset of $K^z \setminus \bigcup_{i=0}^{j-1} K_i^z$ of maximal size for which we can find a single frequency $\alpha_j^z \in \R / \Z$ with $\| p \alpha_j^z - \alpha_x \| < C_1 /\tilde{H}$ for every $((x,\alpha_x),p) \in K_j^z$.

Our goal will be to investigate how the relations between the original frequencies $\alpha_x$ translate into relations between the new frequencies $\alpha_i^z$. To do this we will establish several lemmas in the spirit of the contagion estimates from our previous work \cite{W1} and which encompass the idea of 'parallel rigidity': relations between only a few elements of $K_i^z$ and $K_j^{z'}$ are enough to guarantee that such relations translate to the new frequencies $\alpha_i^z, \alpha_j^{z'}$. 

We begin by recalling a lemma from \cite{W1}. This will allow us to guarantee that if there are many relations between the elements of $\mathcal{A}$, then many of the sets $K_i^z$ we have just constructed will turn out to be quite big.

\begin{lema}
\label{conc}
Let $0 \le \epsilon < \frac{c}{P^2}$ for some sufficiently small $c \gtrsim 1$. Let $S \subseteq \mathcal{P}$ be a set of primes and for each $p \in S$, let $\alpha_p$ be a corresponding element of $\R / \Z$. Let $\sigma > 0$. If for at least $\sigma |S|^2$ pairs $(p_1,p_2) \in S^2$ we have $\| p_1 \alpha_{p_2} - p_2 \alpha_{p_1} \| < \epsilon$, then there exists an element $\alpha \in \R / \Z$ such that $\| p \alpha - \alpha_p \| \lesssim \frac{ \epsilon}{P}$ for $\ge \sigma^2 |S| - O(1)$ primes $p \in S$. 
\end{lema}

\begin{proof}
This is Lemma 2.4 of \cite{W1} except for the explicit dependence on $\sigma$, which can be obtained by performing a simple application of the Cauchy-Schwarz inequality during the proof of that lemma.
\end{proof}

We now explain the context in which the next lemma aims to be used. Given $z, z' \in \mathcal{Z}$ and integers $i,j \ge 0$, suppose we have a pair of elements in $K_i^z$ with prime coordinates $p_1, p_2 \in \mathcal{P}$ and a pair of elements in $K_j^{z'}$ with the same prime coordinates. We wish to show that if the pairs of elements that have the same prime coordinates are related to each other under $\sim$, then a corresponding relation holds for the frequencies $\alpha_i^z, \alpha_j^{z'}$ associated to $K_i^z, K_j^{z'}$.

\begin{lema}
\label{A29}
Let $0 < \epsilon < c/P^2$ for some sufficiently small $c > 0$. Let $\alpha_1, \alpha_2, \beta_1, \ldots, \beta_4 \in \R / \Z$ and let $p_1,p_2,q_1,q_2 \in \mathcal{P}$ be primes with $p_1 \neq p_2$ such that 
$$ \| p_1 \alpha_1 - \beta_1 \| , \| p_2 \alpha_1 - \beta_2 \| , \| p_1 \alpha_2 - \beta_3 \|, \| p_2 \alpha_2 - \beta_4 \| < \epsilon,$$
and
$$ \| q_1 \beta_1 - q_2 \beta_3 \| , \| q_1 \beta_2 - q_2 \beta_4 \| < \epsilon P.$$
Then, $\| q_1 \alpha_1 - q_2 \alpha_2 \| \lesssim \epsilon$.
\end{lema}

\begin{proof}
By our hypotheses, we have that
$$ \| q_1 (p_1 \alpha_1 - \beta_1) \| , \| q_2 (p_1 \alpha_2 - \beta_3) \| < 2 \epsilon P.$$
Since $\| q_1 \beta_1 - q_2 \beta_3 \| < \epsilon P$, it then follows from the triangle inequality that
$$ \| p_1 (q_1 \alpha_1 - q_2 \alpha_2) \| < 5\epsilon P,$$
from where we conclude that
$$ q_1 \alpha_1 - q_2 \alpha_2 = \frac{l_1}{p_1} + O(\epsilon P/p_1) = \frac{l_1}{p_1} + O(\epsilon),$$
for some $l_1 \in \Z$ and as elements of $\R / \Z$. On the other hand, the same argument can be used to derive that
$$ q_1 \alpha_1 - q_2 \alpha_2 =  \frac{l_2}{p_2} + O(\epsilon),$$
for some $l_2 \in \Z$. Therefore,
$$ \frac{l_1}{p_1} - \frac{l_2}{p_2} \equiv O(\epsilon) \, \, (\text{mod }1).$$
Since $\epsilon < c/P^2$ for some sufficiently small $c > 0$, it follows that it must necessarily be $\frac{l_1}{p_1} - \frac{l_2}{p_2} \equiv 0 \, (\text{mod }1)$ and therefore $p_1|l_1$ and $p_2|l_2$, giving the desired result.
\end{proof}

\begin{lema}
\label{Ax2}
Let $\tilde{H} > C P^2$ for some sufficiently large $C \sim 1$. Let $p_1, p_2 \in \mathcal{P}$ be distinct primes and let $(x_i, \beta_i), (y_i, \beta_i')$, $i=1,2$, be elements inside a $(c,\tilde{H})$-configuration in $[Y,2Y]$. Let $(z,\alpha), (z', \alpha') \in \R \times \R / \Z$ be such that for $i=1,2$ we have $|z/p_i - x_i |, |z'/p_i - y_i| \lesssim \tilde{H}$ and $\| p_i \alpha - \beta_i \|, \| p_i \alpha' - \beta_i' \| \lesssim 1/\tilde{H}$. Furthermore, suppose there exist primes $q_1, q_2 \in \mathcal{P}$ with $(x_i, \beta_i) \sim_{q_1,q_2} (y_i, \beta_i')$ for $i=1,2$. Then $| z / q_1 - z' / q_2| \lesssim \tilde{H}$ and $\| q_1 \alpha - q_2 \alpha' \| \lesssim 1/\tilde{H}$.
\end{lema}

\begin{proof}
From the triangle inequality and the relation $(x_1, \beta_1) \sim_{q_1,q_2} (y_1, \beta_1')$, we see that
\begin{align*}
 \frac{1}{p_1} |z/q_1 - z'/q_2| &\le |  (z/p_1)/q_1 - x_1/q_1| + |x_1/q_1 - y_1/q_2 | + |y_1/q_2 - (z'/p_1)/q_2| \\
 &\lesssim \frac{1}{q_1} \tilde{H} + \tilde{H}/P + \frac{1}{q_2} \tilde{H} \lesssim \tilde{H}/P,
 \end{align*}
 which gives us the first claim. The second one follows from our hypotheses and Lemma \ref{A29}.
 \end{proof}
 
 The context in which we wish to apply the next lemma is similar. We are again given $z, z' \in \mathcal{Z}$, integers $i,j \ge 0$, primes $p_1, p_2 \in \mathcal{P}$, a pair of elements in $K_i^z$ whose prime coordinates are $p_1, p_2$ and a pair of elements in $K_j^{z'}$ that also have these prime coordinates. Now, however, we are also given a third element $z'' \in \mathcal{Z}$ and a pair of elements in $K^{z''}$ with these same prime coordinates, but which we do not know to be related to each other under the relation $\sim$ (in particular, we are not assuming they belong to the same set $K_s^{z''}$ in our previous decomposition of $K^{z''}$). The goal of the lemma is to show that if, for $r=1,2$, the element of $K^{z''}$ corresponding to $p_r$ relates under $\sim$ with the elements of $K_i^{z}$ and $K_j^{z'}$ corresponding to $p_r$, then these two elements from $K^{z''}$ must indeed relate with each other under $\sim$. The precise statement is as follows.
 
  \begin{lema}
\label{P11}
Let $\tilde{H} > C P^3$ for some sufficiently large $C \sim 1$. Suppose we are given primes $p_1, p_2, q_1^{(1)}, q_2^{(1)}, q_1^{(2)}, q_2^{(2)} \in \mathcal{P}$ with $q_2^{(1)} \neq q_2^{(2)}$ and elements $(y_{p_1}, \alpha_{p_1}), (y_{p_2}, \alpha_{p_2}), (x^{(1)}_{p_1}, \gamma^{(1)}_{p_1}), (x_{p_2}^{(1)}, \gamma^{(1)}_{p_2})$, $(x_{p_1}^{(2)}, \gamma_{p_1}^{(2)})$, $(x_{p_2}^{(2)}, \gamma_{p_2}^{(2)}) $ in a $(c,\tilde{H})$-configuration, for some $c \sim 1$, such that $|y_{p_1} / p_2 - y_{p_2} / p_1 | \lesssim \tilde{H}/P$, 
 \begin{equation}
 \label{M25A}
  (x^{(j)}_{p_1}, \gamma^{(j)}_{p_1}) \sim_{p_2,p_1} (x_{p_2}^{(j)}, \gamma^{(j)}_{p_2}),
  \end{equation}
 for $j= 1,2$ and
 \begin{equation}
 \label{M25B}
  (y_{p_i}, \alpha_{p_i}) \sim_{q_2^{(j)}, q_1^{(j)}} (x_{p_i}^{(j)}, \gamma_{p_i}^{(j)}),
  \end{equation}
 for every choice of $i,j \in \left\{ 1, 2 \right\}$. Then, $\| p_2 \alpha_{p_1} - p_1 \alpha_{p_2} \| \lesssim P/\tilde{H}$. That is, $(y_{p_1},\alpha_{p_1}) \sim_{p_2,p_1} (y_{p_2},\alpha_{p_2})$.
\end{lema} 

\begin{proof}
From (\ref{M25A}), (\ref{M25B}) and the triangle inequality, we have that
$$ \| p_2 q_2^{(1)} \alpha_{p_1} - p_1 q_2^{(1)} \alpha_{p_2} \| \le \| p_2 q_1^{(1)} \gamma_{p_1}^{(1)} - p_1 q_1^{(1)} \gamma_{p_2}^{(1)} \| +  O(P^2/\tilde{H}) \lesssim P^2 /\tilde{H}. $$
This means that
$$ p_2 \alpha_{p_1} -  p_1 \alpha_{p_2}  \equiv \frac{l_1}{q_2^{(1)}} + O(P/\tilde{H}) \, \, (\text{mod }1),$$
for some $l_1 \in \Z$. Applying the same reasoning, we also have
$$ p_2 \alpha_{p_1} -  p_1 \alpha_{p_2}  \equiv \frac{l_2}{q_2^{(2)}} + O(P/\tilde{H}) \, \, (\text{mod }1),$$
for some $l_2 \in \Z$. The result then follows, since if $C$ is sufficiently large, both estimates can only hold if $q_2^{(1)}|l_1$ and $q_2^{(2)} | l_2$.
\end{proof}

Suppose we are given two sets $K_i^z$, $K_j^z$ in the decomposition of $K^z$. The next lemma shows that if at least two elements of $K_i^z$ relate under $\sim$ with elements of $K_j^z$, then the frequencies $\alpha_i^z, \alpha_j^z$ must be close to each other.

\begin{lema}
\label{A16}
Let $\epsilon < c/P^3$ for a sufficiently small $c \sim 1$. For $i, j \in \left\{ 1, 2 \right\}$, let $\alpha_i, \beta_{i,j} \in \R / \Z$ and let $p_{i,j} \in \mathcal{P}$ be distinct primes. Suppose that for every choice of $i, j$ we have $\| p_{i,j} \alpha_i - \beta_{i,j} \| < \epsilon$ and $\| p_{\overline{i},j} \beta_{i,j} - p_{i,j} \beta_{\overline{i},j} \| \le \epsilon P$, where $\overline{i}=2$ if $i=1$ and $\overline{i}=1$ if $i=2$. Then, $\| \alpha_1 - \alpha_2 \| \lesssim \epsilon / P$.
\end{lema}

\begin{proof}
Combining the hypotheses with the triangle inequality, we see that $\| p_{1,j} p_{2,j} (\alpha_1 - \alpha_2) \| < 5 \epsilon P$ for $j=1,2$. Thus,
$$ \alpha_1 - \alpha_2 = \frac{l_1}{p_{1,1} p_{2,1}} + O(\frac{\epsilon P}{p_{1,1} p_{2,1}}) = \frac{l_2}{p_{1,2} p_{2,2}} + O(\frac{\epsilon P}{p_{1,2} p_{2,2}}),$$
for certain integers $l_1,l_2$. Since the primes $p_{i,j}$ are distinct and belong to $[P,2P]$, the result follows from our assumption on $\epsilon$.
\end{proof}

Lastly, we also have the following variant.

\begin{lema}
\label{gap12}
Let $0 < \epsilon < c/P^2$ for some sufficiently small $c \sim 1$. Let $\beta, \alpha_1, \alpha_2, \alpha_3 \in \R / \Z$ and let $p_1,p_2,p_3 \in \mathcal{P}$ be distinct primes. Suppose that for $i=1,2$ we have $\| p_i \beta - \alpha_i \| < \epsilon$ and $\| p_3 \alpha_i - p_i \alpha_3 \| < \epsilon P$. Then, $\| p_3 \beta - \alpha_3 \| \lesssim \epsilon$.
\end{lema}

\begin{proof}
Let $i=1,2$. By the hypotheses and the triangle inequality, we have that $\| p_3 p_i \beta - p_i \alpha_3 \| < 3 \epsilon P$. It follows that
$$ p_3 \beta - \alpha_3 = \frac{l_i}{p_i} + O ( \epsilon ),$$ 
for some integer $l_i$. The result then follows from our size hypothesis on $\epsilon$.
\end{proof}

\section{Fixed-prime liftings}

Following the strategy outlined in the introduction, our goal in this section will be to establish the following result that decomposes a configuration into several subsets with good behaviour under liftings.

\begin{prop}
\label{lift28}
Let $s=O(1)$ be a positive integer. Let $1 \lesssim c_1, c_2 \le 1$, $1 \lesssim \delta < 1/10^{10}$, $10 \le C_2, B \lesssim 1$,  $P^3 < M \tilde{H} < Y$ for some sufficiently large $M \sim 1$ and assume $K$ is sufficiently large. Let $\mathcal{A}$ be a $(c_1,\tilde{H})$-configuration in $[Y,2Y]$ with $|\mathcal{A}|=c_1 \frac{Y}{\tilde{H}}$ and write $\mathcal{Q}$ for the set of quadruples $((x,\alpha_x), (y,\alpha_y), q_1, q_2)$ in $\mathcal{A}^2 \times \mathcal{P}^2$ with $|x/q_1 - y/q_2|  < \delta \tilde{H}/P$ and $\| q_1 \alpha_x - q_2 \alpha_y \| < C_2 P / \tilde{H}$. Suppose $|\mathcal{Q}| = c_2 |\mathcal{A}||\mathcal{P}|^2$. Then, there exist $p^{\ast} \in \mathcal{P}$, a $(c_1,p^{\ast} \tilde{H})$-configuration $\mathcal{A}^{\uparrow p^{\ast}}$ in $[p^{\ast}Y,2p^{\ast}Y]$ and sets $\mathcal{A}_0, \mathcal{A}_1, \ldots, \mathcal{A}_{s}$ with $\mathcal{A} = \bigcup_{r=0}^{s} \mathcal{A}_r$, such that the following holds:
\begin{enumerate}[(i)]
\item For every $(x,\alpha_x) \in \mathcal{A}$ there exists an element $(p^{\ast} x, \alpha_x^{\uparrow}) \in \mathcal{A}^{\uparrow p^{\ast}}$ with $p^{\ast} \alpha_x^{\uparrow} = \alpha_x$ and all elements of $\mathcal{A}^{\uparrow p^{\ast}}$ are of this form. In particular, $|\mathcal{A}|=| \mathcal{A}^{\uparrow p^{\ast}}|$.
\item $|\mathcal{A}_0| \gtrsim Y/\tilde{H}$, with the implicit constant independent of $s$.
\item For all but $\lesssim \frac{1}{(\log Y)^B} \frac{Y}{\tilde{H}} |\mathcal{P}|^2$ of the quadruples $((x,\alpha_x), (y,\alpha_y), q_1, q_2) \in \mathcal{Q}$ with $(x,\alpha_x) \in \mathcal{A}_r$ for some $0 \le r < s$, we have that $(y,\alpha_y) \in \mathcal{A}_{r+1}$ and $\| q_1 \alpha_x^{\uparrow} - q_2 \alpha_y^{\uparrow} \| < \frac{C_2 P}{p^{\ast} \tilde{H}}$.
\item If $(x,\alpha_x) \in \mathcal{A}_r$ for some $0 \le r < s$, then there exist $\gtrsim \frac{1}{(\log Y)^{rB}} |\mathcal{P}|$ pairs $((y,\alpha_y),q) \in \mathcal{A} \times \mathcal{P}$ with $|(p^{\ast}x)/q-y| \lesssim \tilde{H}$ and $\| q \alpha_x^{\uparrow} - \alpha_y \| \lesssim 1/\tilde{H}$.
\end{enumerate}
\end{prop}

\begin{defi}
If $\mathcal{A}, \mathcal{A}^{\uparrow p^{\ast}}$ are configurations of the type given in the statement of Proposition \ref{lift28}, we say $\mathcal{A}^{\uparrow p^{\ast}}$ is a \emph{lifting} of $\mathcal{A}$.
\end{defi}

For the remainder of this section we will focus on the proof of Proposition \ref{lift28}, so we fix all notation as in its statement. In particular, since we will largely be focusing on quadruples of $\mathcal{Q}$, we will use the following variant of the $\sim$ notation introduced earlier.

\begin{notation}[$\sim$]
\label{2sim}
During this section, given elements $(x,\alpha_x), (y, \alpha_y) \in \mathcal{A}$ and primes $q_1, q_2 \in \mathcal{P}$, we will write $(x,\alpha_x) \sim_{q_1, q_2} (y, \alpha_y)$ to mean that $((x,\alpha_x), (y, \alpha_y), q_1, q_2) \in \mathcal{Q}$ and similarly, we will write $(x,\alpha_x) \sim (y, \alpha_y)$ to mean that there exist $q_1, q_2 \in \mathcal{P}$ with $(x,\alpha_x) \sim_{q_1, q_2} (y, \alpha_y)$. On the other hand, if we write $(x,\alpha_x) \sim_{q_1, q_2}^{C} (y, \alpha_y)$ for an explicit parameter $C$, this will have the same meaning as in Notation \ref{not28}. 
\end{notation}

We now need to define some variants of the sets $K_i^z$ introduced at the beginning of last section. We assume the constant $C_1$ used in the construction of these sets is sufficiently large with respect to $C_2$ and write $C[0]=C_1$. For $1 \le r \le s$, we then let $C[r] \sim 1$ be a constant that is sufficiently large with respect to $C[r-1]$. We write $K_i^z[0]=K_i^z$ and $\alpha_i^z[0]=\alpha_i^z$. Then, for every $1 \le r \le s$, we write $\alpha_0^z[r]=\alpha_0^z$ and let $K_0^z[r]$ consist of those elements $((x,\alpha_x),q) \in K^z$ with $\| q \alpha_0^z[r] - \alpha_x \| < C[r] /\tilde{H}$. Once $\alpha_i^z[r]$ and $K_i^z[r]$ have been defined for all $0 \le i < j$, we then let $K_j^z[r]$ be a subset of maximal size of $K^z \setminus \bigcup_{i=0}^{j-1} K_i^z[r]$ for which we can find a single frequency $\alpha_j^z[r] \in \R / \Z$ such that $\| q \alpha_j^z[r] - \alpha_x \| < C[r] /\tilde{H}$ for every $((x,\alpha_x),q) \in K_j^z[r]$.

Finally, we let $\overline{K}_i^z[0]=K_i^z[0]$ and, for $1 \le r \le s$, we let $\overline{K}_i^z[r]$ consist of those $((x,\alpha_x),q) \in K^z$ that either belong to $K_i^z[r]$ or belong to $K_j^z$ for some $j$ with $\| \alpha_j^z - \alpha_i^z[r] \| \le \overline{C}[r]/(P\tilde{H})$, where $\overline{C}[r] \sim 1$ is a constant that is sufficiently large with respect to $C[r]$. We will also assume that $C[r+1]$ is sufficiently large with respect to $\overline{C}[r]$.

We now proceed to study the sets and frequencies we have just constructed. We begin with the following simple observation.

\begin{lema}
\label{O28}
Every $( (x,\alpha_x) , q) \in K^z$ belongs to $O(1)$ sets $\overline{K}_i^z[r]$.
\end{lema}

\begin{proof}
Since $s=O(1)$, it will suffice to prove this for a fixed choice of $r$. We may also assume $r > 0$, since the case $r=0$ is trivial. Observe now that both the sets $K_i^z$ and the sets $K_i^z[r]$ are partitions of $K^z$. Therefore, we see that identifying $\R / \Z$ with $[0,1)$, it will be enough to show that each subinterval of $\R / \Z$ of size $O(1/(P \tilde{H}))$ contains $O(1)$ elements $\alpha_i^z[r]$. To prove this, observe that otherwise we could find a sequence $\alpha_{i_1}^z[r] \le \alpha_{i_2}^z[r] \le \alpha_{i_3}^z[r]$ inside an interval of size $c/(P \tilde{H})$, for some small $c \sim 1$, and with either $i_1 > i_2 > i_3$ or $i_1 < i_2 < i_3$. We will assume the latter situation holds, since the former case can be handled analogously. But notice that every $((y,\alpha_y),q) \in K_{i_2}^z[r]$ must fail to belong to $K_{i_1}^z[r]$, which by construction means that $\| q \alpha_{i_2}^z[r] - \alpha_y \| < C[r]/ \tilde{H}$ and $\| q \alpha_{i_1}^z[r] - \alpha_y \| \ge C[r]/ \tilde{H}$. If $c$ is sufficiently small, this necessarily implies that $\| q \alpha_{i_3}^z[r] - \alpha_y \| < C[r]/\tilde{H}$. But since $K_{i_3}^z[r]$ is non-empty and must therefore contain an element that does not belong to $K_{i_2}^z[r]$, this contradicts the maximality in the construction of  $K_{i_2}^z[r]$. The result follows.
\end{proof}

\begin{lema}
\label{po17}
Let $z \in \mathcal{Z}$. Let $B_1 \sim 1$ and assume $B_2 \sim 1$ is sufficiently large with respect to $B_1$. Write $\mathcal{U}$ for the set of pairs $(((x,\alpha_x),q_1),((y,\alpha_y),q_2))$ in $K^z \times K^z$ with $\| q_2 \alpha_x - q_1 \alpha_y \| \le B_1 P/\tilde{H}$ and $\| q_2 \alpha_i^z - \alpha_y \| > B_2/\tilde{H}$ for the unique integer $i$ with $((x,\alpha_x),q_1) \in K_i^z$. Then, $|\mathcal{U}|  \lesssim \frac{1}{(\log Y)^{10 s B}} |\mathcal{P}|^2$.
\end{lema}

\begin{proof}
Let $(((x,\alpha_x),q_1),((y,\alpha_y),q_2))$ be an element of $\mathcal{U}$. It is clear that if $B_2$ is sufficiently large then both $((x,\alpha_x),q_1)$ and $((y,\alpha_y),q_2)$ cannot belong to $K_i^z$ for the same $i$. Let then $i \neq j$ and suppose that $( ((x,\alpha_x),q_1) ,((y,\alpha_y),q_2) )$ and $( ((x',\alpha_{x'}),q_1') ,((y',\alpha_{y'}),q_2') )$ are two different elements of $(K_i^z \times K_j^z) \cap \mathcal{U}$. It then follows from Lemma \ref{A16} that, if $B_2$ is sufficiently large with respect to $B_1$, the primes $q_1, q_1', q_2, q_2'$ cannot all be distinct. Furthermore, we see from Lemma \ref{gap12} that it must be $q_2 \neq q_2'$. We conclude that it must be $q_1 = q_1'$, that is, given $i \neq j$ we have that all elements in $(K_i^z \times K_j^z) \cap \mathcal{U}$ must have the same first coordinate $((x,\alpha_x),q) \in K_i$. 

Let $j_0$ be the largest integer with $|K_{j_0}^z| \ge \frac{1}{(\log Y)^{100 sB}} |\mathcal{P}|$. In particular, $j_0 \le (\log Y)^{100sB}$. By the previous observation, we have that
$$ \sum_{i,j : 0 \le i \le j_0} | (K_i^z \times K_j^z) \cap \mathcal{U}| \le (j_0+1) \sum_{j \in \N} |K_j^z| \lesssim (\log Y)^{100sB} |\mathcal{P}|,$$
which is acceptable if $K$ is sufficiently large. Let us write $\mathcal{X} = \bigcup_{i > j_0} K_i^z$. Let $j \le j_0$ and write $\mathcal{X}_j$ for those elements $((x,\alpha_x),q) \in \mathcal{X}$ with $| (\left\{ ((x,\alpha_x),q) \right\} \times K_j) \cap \mathcal{U}| > 1$. From Lemma \ref{gap12}, we deduce that $\| q \alpha_j^z - \alpha_x \| \lesssim 1/\tilde{H}$ for every $((x,\alpha_x),q) \in \mathcal{X}_j$. By the pigeonhole principle, this implies in turn that we can find some $\tilde{\alpha}$ with $\| q \tilde{\alpha} - \alpha_x \| < C_1/\tilde{H}$ for $\gtrsim |\mathcal{X}_j|$ elements of $\mathcal{X}_j$, with $C_1$ as in the construction of the sets $K_i^z$. From the definition of $\mathcal{X}$ and the maximality of the sets $K_i^z$, we deduce that $|\mathcal{X}_j| \lesssim \frac{ |\mathcal{P}|}{(\log Y)^{100 sB}}$ and therefore
$$ \sum_{j \le j_0} | (\mathcal{X} \times K_j^z) \cap \mathcal{U} | \le \sum_{j \le j_0} \left( |\mathcal{X}_j||K_j^z|+|\mathcal{X}| \right) \lesssim |\mathcal{P}| \left( \frac{ |\mathcal{P}|}{(\log Y)^{100 sB}}+j_0 \right),$$
which is again acceptable. It only remains to bound $|(\mathcal{X} \times \mathcal{X}) \cap \mathcal{U}|$. For this, it will suffice to show that there are only $\lesssim \frac{1}{(\log Y)^{10sB}} |\mathcal{P}|^2$ pairs of elements $(((x,\alpha_x),q_1),((y,\alpha_y),q_2)) \in \mathcal{X} \times \mathcal{X}$ with $\| q_2 \alpha_x - q_1 \alpha_y \| \lesssim P/\tilde{H}$. Indeed, we could otherwise use Lemma \ref{conc} to obtain a frequency $\tilde{\alpha}$ with $\| q \tilde{\alpha} - \alpha_x \| \lesssim 1/\tilde{H}$ for $\gtrsim \frac{1}{(\log Y)^{20sB}} |\mathcal{P}|$ elements $( (x,\alpha_x), q) \in \mathcal{X}$. By the same pigeonhole argument as before, this implies that we can find some $\tilde{\alpha}'$ such that $\| q \tilde{\alpha}' - \alpha_x \| < C_1/\tilde{H}$ for $\gtrsim \frac{1}{(\log Y)^{20sB}} |\mathcal{P}|$ of these elements. But since this would contradict the choice of $j_0$, we obtain the result.
\end{proof}

\begin{notation}
Given $(x,\alpha_x) \in \mathcal{A}$, we write $\mathcal{P}_x$ for the set of pairs $(q_1,q_2) \in \mathcal{P}^2$ for which there exists some $(y,\alpha_y) \in \mathcal{A}$ with $x \sim_{q_1,q_2} y$. 
\end{notation}

\begin{lema}
\label{T18}
Let $\mathcal{T}$ be the set of tuples $(i,r,(x,\alpha_x),z,p,q_1,q_2) \in \N_0^2 \times \mathcal{A} \times \mathcal{Z} \times \mathcal{P}^3$ such that $((x,\alpha_x),p) \in \overline{K}_i^z[r]$, $(q_1,q_2) \in \mathcal{P}_x$ and there are at most $\frac{1}{(\log Y)^{B}} |\overline{K}_i^z[r]|$ other elements $((x',\alpha_{x'}),p') \in \overline{K}_i^z[r]$ with $(q_1,q_2) \in \mathcal{P}_{x'}$ and $|px-p'x'| \le \frac{P\tilde{H}}{10^{10}} $. Then $|\mathcal{T}| \lesssim \frac{1}{(\log Y)^{B}} \frac{Y}{\tilde{H}} |\mathcal{P}|^3$.
\end{lema}

\begin{proof}
By definition of $K^z$ we know that every $((x,\alpha_x),p) \in K^z$ satisfies $|px -z | < P\tilde{H}/50$. For a fixed choice of $i,z,r,q_1,q_2$, it then follows by a simple pigeonholing argument that there can be at most $\lesssim \frac{1}{(\log Y)^{B}} |\overline{K}_i^z[r]|$ tuples $(i,r,(x,\alpha_x),z,p,q_1,q_2) \in \mathcal{T}$. The result then follows upon summing among all choices of $i,z,r,q_1,q_2$, using Lemma \ref{O28}.
\end{proof} 

\begin{lema}
\label{quad17}
Let $C_3 \sim 1$ and assume $C_4 \sim 1$ is sufficiently large with respect to $C_3$. Then, there are $\lesssim \frac{Y}{\tilde{H}} |\mathcal{P}|^3$ quadruples $((x_1,\alpha_{x_1}),p_1)$,$ ((x_2, \alpha_{x_2}),p_2)$, $((y_1, \alpha_{y_1}),p_1)$, $((y_2, \alpha_{y_2}),p_2)$ in $\bigcup_{z \in \mathcal{Z}} K^z$ such that $(x_1,\alpha_{x_1}) \sim_{p_2,p_1}^{C_3} (x_2, \alpha_{x_2})$, $(x_i, \alpha_{x_i}) \sim_{q_1,q_2} (y_i,\alpha_{y_i})$ for $i=1,2$ and some $q_1, q_2 \in \mathcal{P}$, but $(y_1,\alpha_{y_1}) \nsim_{p_2,p_1}^{C_4} (y_2,\alpha_{y_2})$. 
\end{lema}

\begin{proof}
Let $((y_1,\alpha_{y_1}),p_1), ((y_2,\alpha_{y_2}),p_2)$ be a pair belonging to such a quadruple. Applying the triangle inequality, we see that the hypotheses imply that $|p_1 y_1 - p_2 y_2| \lesssim P\tilde{H}$, from where it follows that $((y_1,\alpha_{y_1}),p_1) \in K^z$ and $((y_2,\alpha_{y_2}),p_2) \in K^{z'}$ for some $z, z' \in \mathcal{Z}$ with $|z-z'|=O(P \tilde{H})$. That is, once $z$ is fixed, there are $O(1)$ possibilities for $z'$. By Lemma \ref{P11} we also see that, as long as $C_4$ is sufficiently large, any such quadruple this pair belongs to must give rise to pairs of primes $q_1, q_2 \in \mathcal{P}$ with the same value of $q_2$. But this implies that there are at most $|\mathcal{P}|$ such quadruples. The result then follows, since $\sum_{z \in \mathcal{Z}} | K^z| |\mathcal{P}|^2 \lesssim \frac{Y}{\tilde{H}} |\mathcal{P}|^3$. 
\end{proof}

\begin{lema}
\label{019}
We have $\sum_{z \in \mathcal{Z}} |K_0^z| \gtrsim \frac{Y}{\tilde{H}} |\mathcal{P}|$, with the implicit constant independent of $s$.
\end{lema}

\begin{proof}
For each pair $(x,\alpha_x),(y,\alpha_y) \in \mathcal{A}$ with $(x,\alpha_x) \sim_{p,q} (y,\alpha_y)$ we can find some $z \in \mathcal{Z}$ with $|qx-z|, |py - z| < 4 \delta P\tilde{H}$ and thus in particular $((x,\alpha_x),q),((y,\alpha_y),p) \in K^z$. Since 
$$\sum_{(x,\alpha_x) \in \mathcal{A}} | \left\{ (y,\alpha_y) \in \mathcal{A} : (x,\alpha_x) \sim (y,\alpha_y) \right\} | \gtrsim \frac{Y}{\tilde{H}} |\mathcal{P}|^2,$$
this allows us to find $\gtrsim \frac{Y}{\tilde{H}}$ elements $z \in \mathcal{Z}$ for which there exist $\gtrsim |\mathcal{P}|^2$ pairs $((x,\alpha_x),q),((y,\alpha_y),p) \in K^z$ with $(x,\alpha_x) \sim_{p,q} (y,\alpha_y)$ and $|qx-z|, |py - z| < 4 \delta P\tilde{H}$. As long as $C_1$ is sufficiently large with respect to $C_2$, it follows from Lemma \ref{conc} and the definition of $K_0^z$ that for each such $z$ it must be $|K_0^z| \gtrsim |\mathcal{P}|$ and this gives the result.
\end{proof}

\begin{coro}
\label{count30}
There exist $\gtrsim |\mathcal{P}|$ primes $p \in \mathcal{P}$ which are second coordinates of $K_0^z$ for $\gtrsim Y/\tilde{H}$ choices of $z \in \mathcal{Z}$ with $|K_0^z| \gtrsim |\mathcal{P}|$, with the implicit constants independent of $s$.
\end{coro}

\begin{proof}
This is immediate from Lemma \ref{019}.
\end{proof}

We now come to the main estimate in the proof of Proposition \ref{lift28}.

\begin{lema}
\label{privado19}
For all but at most $\lesssim \frac{1}{(\log Y)^{B}} \frac{Y}{\tilde{H}} |\mathcal{P}|^3$ choices of $(x,\alpha_x),(y,\alpha_y) \in \mathcal{A}$, $z,z' \in \mathcal{Z}$ and $p \in \mathcal{P}$ with $((x,\alpha_x),p) \in \overline{K}_{{i}}^z[r]$, $r < s$, $|\overline{K}_{{i}}^z[r]| \gtrsim (\log Y)^{-sB}|\mathcal{P}|$, $((y,\alpha_y),p) \in K_j^{z'}$, $|py-z'| < 4 \delta P\tilde{H}$ and $(x,\alpha_x) \sim_{q_1,q_2} (y,\alpha_y)$ for some primes $q_1,q_2 \in \mathcal{P}$, the following holds true:
\begin{enumerate}[(i)]
\item $\| q_1 \alpha_{{i}}^z[r] - q_2 \alpha_j^{z'} \| \lesssim 1/\tilde{H}$,
\item there exists $\overline{j}$ with $((y,\alpha_y),p) \in \overline{K}_{\overline{j}}^{z'}[r+1]$, such that $\| \alpha_j^{z'} - \alpha_{\overline{j}}^{z'}[r+1] \| \lesssim 1/(\tilde{H}P)$ and $|\overline{K}_{\overline{j}}^{z'}[r+1]| \gtrsim \frac{1}{(\log Y)^B} |\overline{K}_{{i}}^z[r]|$.
\end{enumerate}
\end{lema}

\begin{proof}
We begin by noticing that once $(x,\alpha_x), z, q_1, q_2$ are fixed, there is at most one choice for $(y,\alpha_y)$ and $p$ and $O(1)$ choices for $z'$. It then follows from Lemma \ref{T18} that we may restrict attention to those cases where $(q_1,q_2) \in \mathcal{P}_{x'}$ for $\gtrsim \frac{1}{(\log Y)^B} |\overline{K}_i^z[r]|$ other elements $((x', \alpha_{x'}),p') \in \overline{K}_i^z[r]$ with $|px - p'x'| < \frac{P\tilde{H}}{10^{10}}$. Each such $(x', \alpha_{x'}) $ then gives rise to a quadruple $((x,\alpha_x),p)$, $((x',\alpha_{x'}),p')$, $((y,\alpha_y),p)$, $((y',\alpha_{y'}),p')$ with $(x,\alpha_x) \sim_{q_1,q_2} (y,\alpha_y)$, $(x',\alpha_{x'}) \sim_{q_1,q_2} (y',\alpha_{y'})$ and $(x,\alpha_x) \sim_{p',p}^{O(\overline{C}[r])} (x',\alpha_{x'})$. By Lemma \ref{quad17} and our lower bound on $|\overline{K}_{{i}}^z[r]|$, if $C_4 \sim 1$ is sufficiently large with respect to $\overline{C}[r]$, there are $\lesssim  \frac{Y}{\tilde{H}}  |\mathcal{P}|^2 (\log Y)^{(s+1)B}$ choices of $(x,\alpha_x),z,q_1,q_2$ for which at least half of these quadruples satisfy $(y,\alpha_y) \nsim_{p',p}^{C_4} (y',\alpha_{y'})$. Hence, it will suffice to consider those $(x,\alpha_x),z,q_1,q_2$ for which at least half of these quadruples satisfy $(y,\alpha_y) \sim_{p',p}^{C_4} (y',\alpha_{y'})$. Notice that for each such quadruple we have $|xpq_2/q_1 - yp | , |x' p' q_2/q_1 - y' p' | \le 4 \delta P\tilde{H}$. It then follows from the triangle inequality that $|py - p'y'| < P \tilde{H}/10^4$, say. In particular, since $|py-z'| < 4 \delta P\tilde{H}$, this implies that $((y',\alpha_{y'}),p') \in K^{z'}$. If we have $\gtrsim  \frac{1}{(\log Y)^B} \frac{Y}{\tilde{H}} |\mathcal{P}|^3$ choices of $(x,\alpha_x),(y,\alpha_y),z,z'$ for which at least half of these elements $((y',\alpha_{y'}),p')$ satisfy $\| p' \alpha_j^{z'} - \alpha_{y'} \| > C_5 /\tilde{H}$ for some $C_5 \sim 1$ that is sufficiently large with respect to $C_4$, this leads to $\gtrsim \frac{1}{(\log Y)^{(s+2)B}} \frac{Y}{\tilde{H}} |\mathcal{P}|^4$ pairs $((y,\alpha_y),p), ((y',\alpha_{y'}),p')$ belonging to the exceptional pairs in Lemma \ref{po17}, counting repetitions among the choices of $(x,\alpha_x),(y,\alpha_y),z,z'$. But since each such pair $((y,\alpha_y),p), ((y',\alpha_{y'}),p')$ can be repeated at most $\lesssim |\mathcal{P}|^2$ times, corresponding to the possible choices of $(x,\alpha_x), z, z'$ with $(x,\alpha_x) \sim (y,\alpha_y)$, this gives us a contradiction once $Y$ is sufficiently large.

We conclude that after removing $\lesssim  \frac{1}{(\log Y)^B} \frac{Y}{\tilde{H}} |\mathcal{P}|^3$ cases, for each remaining choice of $(x,\alpha_x),(y,\alpha_y),z,z'$ we have $\gtrsim \frac{1}{(\log Y)^B} |\overline{K}_i^z[r]|$ elements $((y',\alpha_{y'}),p') \in K^{z'}$ with $\| p' \alpha_j^{z'} - \alpha_{y'} \| \lesssim 1/\tilde{H}$, with the implicit constant depending on $\overline{C}[r]$. This necessarily implies, as long as $C[r+1]$ is sufficiently large with respect to $\overline{C}[r]$ and by the maximality in the construction of the sets $K_u^{z'}[r+1]$, that at least half of these $((y',\alpha_{y'}),p')$ belong to $\bigcup_u^{'} K_u^{z'}[r+1]$, where the union is restricted to those $u$ with with $|K_u^{z'}[r+1]| \gtrsim  \frac{1}{(\log Y)^B} |\overline{K}_i^z[r]| \gtrsim \frac{1}{(\log Y)^{(s+1)B}} |\mathcal{P}|$. Since there are $\lesssim (\log Y)^{(s+1)B}$ such sets, this crudely means that we can find two such elements $((y',\alpha_{y'}),p'), ((y'',\alpha_{y''}),p'')$ belonging to the same $K_u^{z'}[r+1]$. We then have that $\| p' \alpha_u^{z'}[r+1] - \alpha_{y'} \|, \| p'' \alpha_u^{z'}[r+1] - \alpha_{y''} \| < C[r+1]/\tilde{H}$. Since we also know that $\| p' \alpha_j^{z'} - \alpha_{y'} \| , \| p'' \alpha_j^{z'} - \alpha_{y''} \| \lesssim 1/\tilde{H}$ and the primes $p', p''$ are distinct, it follows from the triangle inequality that $\| \alpha_u^{z'}[r+1] - \alpha_j^{z'} \| <\overline{C}[r+1] /(P\tilde{H})$, as long as $\overline{C}[r+1]$ was chosen sufficiently large with respect to $C[r+1]$. We thus conclude that $((y,\alpha_y),p) \in \overline{K}_u^z[r+1]$. Writing $\overline{j}=u$, this gives (ii). Finally, if $((x',\alpha_{x'}),p'), ((x'',\alpha_{x''}),p'')$ are the elements of $\overline{K}_i^z[r]$ that relate to these elements $((y',\alpha_{y'}),p'), ((y'',\alpha_{y''}),p'')$, we may apply Lemma \ref{A29} to the quadruple $((x',\alpha_{x'}),p'), ((x'',\alpha_{x''}),p''), ((y',\alpha_{y'}),p'), ((y'',\alpha_{y''}),p'')$ to deduce that $\| q_1 \alpha_{{i}}^z[r] - q_2 \alpha_{\overline{j}}^{z'}[r+1] \| \lesssim 1/\tilde{H}$. Claim (i) then follows from the triangle inequality. 
\end{proof}

We immediately obtain the following corollary.

\begin{coro}
\label{p19}
There exists some $p^{\ast} \in \mathcal{P}$ satisfying the conclusion of Corollary \ref{count30} and such that the conclusions of Lemma \ref{privado19} hold, with $p=p^{\ast}$, for all but at most $\lesssim \frac{1}{(\log Y)^B} \frac{Y}{\tilde{H}} |\mathcal{P}|^2$ choices of $(x,\alpha_x),(y,\alpha_y),z,z'$ as in the statement of that lemma.
\end{coro}

We are now ready to prove Proposition \ref{lift28}.

\begin{proof}[Proof of Proposition \ref{lift28}]
Let us fix a choice of $p^{\ast}$ as in Corollary \ref{p19}. To each $(x,\alpha_x) \in \mathcal{A}$ we associate a pair $(i_x,z_x) \in \N_0 \times \mathcal{Z}$ in the following way. If $((x,\alpha_x),p^{\ast}) \in K_0^z$ for some $z \in \mathcal{Z}$, we take $i_x=0$ and let $z_x$ be the choice of $z$ that maximises $|K_0^z|$ among all $K_0^z$ that contain $((x,\alpha_x),p^{\ast})$. For the remaining $(x,\alpha_x) \in \mathcal{A}$, we then let $z_x$ be an element in $\mathcal{Z}$ closest to $p^{\ast}x$ and let $i_x$ be the unique choice such that $((x,\alpha_x),p^{\ast}) \in K_{i_x}^{z_x}$. Notice in particular that, for every $(x,\alpha_x) \in \mathcal{A}$, we have $|p^{\ast}x-z_x| < 4 \delta P \tilde{H}$.

Next, we let $\mathcal{A}_0$ consist of those $(x,\alpha_x) \in \mathcal{A}$ with $|K_{i_x}^{z_x}| \gtrsim |\mathcal{P}|$. By Corollary \ref{count30} and our choice of $p^{\ast}$, we have that $|\mathcal{A}_0| \gtrsim Y/\tilde{H}$. Then, for every $1 \le r < s$, we let $\mathcal{A}_r$ consist of those $(x,\alpha_x) \in \mathcal{A}$ such that $((x,\alpha_x),p^{\ast}) \in \overline{K}_i^{z_x}[r]$ for some $i$ with $|\overline{K}_i^{z_x}[r]| \gtrsim \frac{1}{(\log Y)^{rB}} |\mathcal{P}|$ and $\| \alpha_i^{z_x}[r] - \alpha_{i_x}^{z_x} \| \lesssim 1/(\tilde{H}P)$. We then let $\mathcal{A}_s = \mathcal{A} $.

Finally, given $(x,\alpha_x) \in \mathcal{A}$ we let $\alpha_x^{\uparrow} \in \R / \Z$ be the unique pre-image of $\alpha_x$ by $p^{\ast}$ lying at distance $O(1/(\tilde{H}P))$ from $\alpha_{i_x}^{z_x}$. Notice that if $x \in \mathcal{A}_r$, $0 \le r < s$, there will be $\gtrsim \frac{1}{(\log Y)^{r B}} |\mathcal{P}|$ different elements $((y,\alpha_{y}),q) \in \mathcal{A} \times \mathcal{P}$ with $|p^{\ast}x/q - y| \lesssim \tilde{H}$ and $\| q \alpha_x^{\uparrow} - \alpha_{y} \| \lesssim 1/\tilde{H}$.

We let $\mathcal{A}^{\uparrow p^{\ast}}$ consist of the pairs of the form $(p^{\ast} x, \alpha_x^{\uparrow})$, with $(x,\alpha_x) \in \mathcal{A}$. In order to finish the proof of Proposition \ref{lift28}, it only remains to verify that the last claim of (iii) holds. For this, observe that after removing the $\lesssim \frac{1}{(\log Y)^B} \frac{Y}{\tilde{H}} |\mathcal{P}|^2$ exceptions in the statement of Corollary \ref{p19}, we have from the definitions, Lemma \ref{privado19} and the triangle inequality that the remaining choices of $(x,\alpha_x) \sim_{q_1, q_2} (y,\alpha_y)$, with $(x,\alpha_x) \in \mathcal{A}_r$ for some $0 \le r < s$, satisfy $\| q_1 \alpha_x^{\uparrow} - q_2 \alpha_y^{\uparrow} \| \lesssim 1/\tilde{H}$. But we also have that $\| p^{\ast}(q_1 \alpha_x^{\uparrow} - q_2 \alpha_y^{\uparrow}) \| = \| q_1 \alpha_x - q_2 \alpha_y \| < C_2 P /\tilde{H}$. The result then follows. 
\end{proof}

\section{Concentration increment argument}

Throughout this section we will let the hypotheses and notation be as in Proposition \ref{lift28} and Notation \ref{2sim}. We will also assume that $\log Y < P^{\epsilon}$ for some sufficiently small $\epsilon > 0$. Our goal will be to carry out the concentration increment argument discussed in the introduction. Unlike the last two sections, we will now start relying on Corollary \ref{ML0} and therefore on RH. We begin by formally defining the concentration parameter $\mathcal{C}(\mathcal{B})$ involved.

\begin{defi}[$\mathcal{R}, \mathcal{C}$]
Given subsets $\mathcal{B}_1, \mathcal{B}_2$ of $\mathcal{A}$, we shall write $\mathcal{R}(\mathcal{B}_1,\mathcal{B}_2)$ for the number of quadruples $( (x,\alpha_x), (y,\alpha_y),q_1,q_2) \in \mathcal{Q}$ with $(x,\alpha_x) \in \mathcal{B}_1$ and $(y,\alpha_y) \in \mathcal{B}_2$. Also, given a subset $\mathcal{B} \subseteq \mathcal{A}$, we will abbreviate $\mathcal{R}(\mathcal{B}) = \mathcal{R} (\mathcal{B},\mathcal{B})$ and define
$$ \mathcal{C}(\mathcal{B}) = \frac{Y}{\tilde{H}} \frac{\mathcal{R}(\mathcal{B})}{|\mathcal{P}|^2 |\mathcal{B}|^2}.$$
\end{defi}

We have the following basic estimate on this parameter.

\begin{lema}
\label{C1}
Let $\mathcal{B}$ be a subset of $\mathcal{A}$ with $|\mathcal{B}| \gtrsim \frac{1}{(\log Y)^B} \frac{Y}{\tilde{H}}$. Then, if $K$ is sufficiently large with respect to $B$, we have $\mathcal{C}(\mathcal{B}) = O(1)$ for some absolute implicit constant.
\end{lema}

\begin{proof}
This is immediate from Corollary \ref{ML0}.
\end{proof}

\begin{lema}
\label{voy31}
There exists some $0 \le t < s$ such that
$$ \mathcal{R} ( \bigcup_{r=0}^t \mathcal{A}_r, \mathcal{A} \setminus  \bigcup_{r=0}^t \mathcal{A}_r) \lesssim \max \left\{ \frac{|\mathcal{P}|^2}{(\log Y)^B} \frac{Y}{\tilde{H}}, \frac{ |\mathcal{P}|^2}{s} |\mathcal{A} \setminus  \bigcup_{r=0}^t \mathcal{A}_r| \right\}.$$
\end{lema}

\begin{proof}
Let $0 \le t < s$ be such that $|\mathcal{A}_t \setminus \bigcup_{r=0}^{t-1} \mathcal{A}_r| \lesssim (s \tilde{H})^{-1} Y$. By Proposition \ref{lift28} (iii), we already know that 
$$\mathcal{R}( \bigcup_{r=0}^{t-1} \mathcal{A}_r,  \mathcal{A} \setminus  \bigcup_{r=0}^t \mathcal{A}_r) \lesssim  \frac{|\mathcal{P}|^2}{(\log Y)^B} \frac{Y}{\tilde{H}}.$$
But from Corollary \ref{ML0}, we have that 
$$ \mathcal{R} ( \mathcal{A}_t  \setminus \bigcup_{r=0}^{t-1} \mathcal{A}_r , \mathcal{A} \setminus  \bigcup_{r=0}^t \mathcal{A}_r) \lesssim  \frac{|\mathcal{P}|^2}{s} |\mathcal{A} \setminus  \bigcup_{r=0}^t \mathcal{A}_r| + P^{2-c_0}  Y / \tilde{H}.$$
The result then follows provided $K$ is sufficiently large with respect to $B$ and $c_0$.
\end{proof}

\begin{lema}
\label{bday1}
Assume $s$ is sufficiently large. If $ |\mathcal{A} \setminus  \bigcup_{r=0}^{s-1} \mathcal{A}_r| \gtrsim \frac{1}{(\log Y)^{B/2}} \frac{Y}{\tilde{H}}$, there exists a proper subset $\mathcal{B} \subseteq \mathcal{A}$ with $|\mathcal{B}| \gtrsim \frac{1}{(\log Y)^{B/2}} \frac{Y}{\tilde{H}}$ and $\mathcal{C}(\mathcal{B}) \ge \mathcal{C}(\mathcal{A}) \left( \frac{|\mathcal{B}|}{|\mathcal{A}|} \right)^{-1/2}$.
\end{lema}

\begin{proof}
By Lemma \ref{voy31} we know there exists some $0 \le t < s$ and $c \sim 1$ with 
$$ \mathcal{R} ( \bigcup_{r=0}^t \mathcal{A}_r, \mathcal{A} \setminus  \bigcup_{r=0}^t \mathcal{A}_r) \le c \frac{ |\mathcal{P}|^2}{s} |\mathcal{A} \setminus  \bigcup_{r=0}^t \mathcal{A}_r| .$$
Let us fix such a choice of $t$ and write $\mathcal{B}_1 = \bigcup_{r=0}^t \mathcal{A}_r$ and $\mathcal{B}_2 = \mathcal{A} \setminus  \bigcup_{r=0}^t \mathcal{A}_r$. Let us also write $\mu_1 = |\mathcal{B}_1|/|\mathcal{A}|$ and $\mu_2 = |\mathcal{B}_2|/|\mathcal{A}|$, so in particular, $\mu_1 + \mu_2 = 1$. We will show the result holds with either $\mathcal{B}=\mathcal{B}_1$ or $\mathcal{B}=\mathcal{B}_2$. Observe that 
$$ \mathcal{R}(\mathcal{B}_1, \mathcal{A}) + \mathcal{R}(\mathcal{B}_2, \mathcal{A}) = |\mathcal{Q}| = c_2 |\mathcal{P}|^2  |\mathcal{A}| .$$
Suppose $\mathcal{R}(\mathcal{B}_1, \mathcal{A}) \ge c_2 |\mathcal{P}|^2 | \mathcal{B}_1 |$. We then have that
$$ \mathcal{C}(\mathcal{B}_1) \ge \frac{Y}{\tilde{H}} \frac{ c_2 |\mathcal{B}_1| - c s^{-1} |\mathcal{B}_2|}{|\mathcal{B}_1|^2} = \frac{c_2}{c_1} \frac{1- c(1-\mu_1) (c_2s\mu_1)^{-1}}{\mu_1}.$$
Since $\mathcal{A}_0 \subseteq \mathcal{B}_1$, we know that $\mu_1 \gtrsim 1$ with the implicit constant independent of $s$. This means that if $s$ is chosen sufficiently large, we have $\frac{\mu_1}{1 + \mu_1^{1/2}} \ge c/(s c_2)$ and therefore
$$ \mathcal{C}(\mathcal{B}_1) \ge \frac{c_2}{c_1} \mu_1^{-1/2} = \mathcal{C}(\mathcal{A}) \left( \frac{|\mathcal{B}_1|}{|\mathcal{A}|} \right)^{-1/2}.$$
We may thus assume that $\mathcal{R}(\mathcal{B}_2, \mathcal{A}) \ge c_2 |\mathcal{P}|^2 | \mathcal{B}_2 |$. Then, 
$$ \mathcal{C}(\mathcal{B}_2) \ge \frac{Y}{\tilde{H}} \frac{ c_2 |\mathcal{B}_2| - cs^{-1} |\mathcal{B}_2|}{|\mathcal{B}_2|^2}  = \frac{c_2}{c_1} \frac{1-c (sc_2)^{-1}}{\mu_2}.$$
Since $\mu_1 \gtrsim 1$ and $\mu_2 = 1 - \mu_1$, we see that $s$ may be chosen sufficiently large as to guarantee that $c (sc_2)^{-1} < 1-\mu_2^{1/2}$. The result then follows, since this implies that
$$ \mathcal{C}(\mathcal{B}_2) \ge \frac{c_2}{c_1} \mu_2^{-1/2} = \mathcal{C}(\mathcal{A}) \left( \frac{|\mathcal{B}_2|}{|\mathcal{A}|} \right)^{-1/2}.$$
\end{proof}

We can now show that there exists some dense subset $\mathcal{B} \subseteq \mathcal{A}$ whose relations lift almost surely through a fixed prime $p^{\ast}$.

\begin{lema}
\label{mau1}
Assume $s$ is sufficiently large. Then, there exists a subset $\mathcal{B} \subseteq \mathcal{A}$ with $|\mathcal{B}| \gtrsim Y/\tilde{H}$ and $\mathcal{C}(\mathcal{B}) \ge \mathcal{C}(\mathcal{A}) \left( \frac{|\mathcal{B}|}{|\mathcal{A}|} \right)^{-1/2}$, a prime $p^{\ast} \in \mathcal{P}$ and a lifting $\mathcal{B}^{\uparrow p^{\ast}}$ of $\mathcal{B}$, such that the following holds:
\begin{enumerate}[(i)]
\item for all but $\lesssim \frac{|\mathcal{P}|^2}{(\log Y)^{B/2}} \frac{Y}{\tilde{H}}$ of the pairs $(x,\alpha_x), (y,\alpha_y) \in \mathcal{B}$ with $(x,\alpha_x) \sim_{q_1,q_2} (y,\alpha_y)$ for some $q_1, q_2 \in \mathcal{P}$, the corresponding elements $(p^{\ast} x,\alpha_x^{\uparrow}), (p^{\ast} y, \alpha_y^{\uparrow}) \in \mathcal{B}^{\uparrow p^{\ast}}$ satisfy $\| q_1 \alpha_x^{\uparrow} - q_2 \alpha_y^{\uparrow} \| < \frac{C_2 P}{p^{\ast} \tilde{H}}$,
\item there exists $\tilde{c} \sim 1$, which can be made arbitrarily small provided $K$ is sufficiently large, such that for all but $\lesssim \frac{1}{(\log Y)^{B/2}} \frac{Y}{\tilde{H}}$ of the elements $(p^{\ast} x, \alpha_x^{\uparrow}) \in \mathcal{B}^{\uparrow p^{\ast}}$, there are $\gtrsim P^{1-\tilde{c}}$ pairs $((y,\alpha_y),q) \in \mathcal{B} \times \mathcal{P}$ with $|(p^{\ast}x)/q-y| \lesssim \tilde{H}$ and $\| q \alpha_x^{\uparrow} - \alpha_y \| \lesssim 1/\tilde{H}$.
\end{enumerate}
\end{lema}

\begin{proof}
Notice that if we are given a sequence $\mathcal{A}=\mathcal{B}^{(1)} \supseteq \ldots \supseteq \mathcal{B}^{(m)}$ with $\mathcal{C}(\mathcal{B}^{(i+1)}) \ge \mathcal{C}(\mathcal{B}^{(i)}) \left( \frac{|\mathcal{B}^{(i+1)}|}{|\mathcal{B}^{(i)}|} \right)^{-1/2}$ for every $1 \le i < m$, then $\mathcal{C} (\mathcal{B}^{(m)}) \ge \mathcal{C}(\mathcal{B}^{(1)}) \left( \frac{|\mathcal{B}^{(m)}|}{|\mathcal{B}^{(1)}|} \right)^{-1/2}$. Furthermore, if $|\mathcal{B}^{(m)}| \gtrsim \frac{1}{(\log Y)^{B}} \frac{Y}{\tilde{H}}$, it follows from Lemma \ref{C1} that $\mathcal{C}(\mathcal{B}^{(m)}) = O(1)$. In particular, since $\mathcal{C}(\mathcal{B}^{(1)}) \sim 1$, it must be $|\mathcal{B}^{(m)}| \gtrsim |\mathcal{B}^{(1)}|$. It thus follows that after iterating Lemma \ref{bday1} finitely many times, where at each step we apply Proposition \ref{lift28} to the subset that we obtain, we arrive at some $\mathcal{B} \subseteq \mathcal{A}$ with $\mathcal{C}(\mathcal{B}) \ge \mathcal{C} (\mathcal{A} ) \left( \frac{|\mathcal{B}|}{|\mathcal{A}|} \right)^{-1/2}$, $|\mathcal{B}| \gtrsim |\mathcal{A}| \gtrsim Y/\tilde{H}$ and with $|\mathcal{B} \setminus \bigcup_{r=0}^{s-1} \mathcal{B}_i| \lesssim \frac{1}{(\log Y)^{B/2}} \frac{Y}{\tilde{H}}$, where now the sets $\mathcal{B}_0,\ldots,\mathcal{B}_s$ are the ones obtained from applying Proposition \ref{lift28} with $\mathcal{B}$ in place of $\mathcal{A}$. The result then follows from items (iii) and (iv) of that proposition.
\end{proof}

We can now iterate this to obtain a sequence of configurations at different scales related through a chain of concentration inequalities.

\begin{lema}
\label{uni2}
Let $k \lesssim \log Y$. Then, there exist a sequence of primes $p_1^{\ast}, \ldots, p_k^{\ast}$, $(c, \tilde{H} \prod_{i=1}^j p_i^{\ast} )$-configurations $\mathcal{A}_{(j)}$ in $[Y\prod_{i=1}^j p_i^{\ast} , 2Y \prod_{i=1}^j p_i^{\ast} ]$, for $0 \le j \le k$ and some $c \sim 1$ and corresponding liftings $\mathcal{A}_{(j)}^{\uparrow p_{j+1}^{\ast}}$ of $\mathcal{A}_{(j)}$ for every $0 \le j < k$, such that the following holds:
\begin{enumerate}[(i)]
\item $\mathcal{A}_{(0)} \subseteq \mathcal{A}$ and $\mathcal{A}_{(j)} \subseteq \mathcal{A}_{(j-1)}^{\uparrow p_j^{\ast}}$ for every $1 \le j \le k$,
\item $|\mathcal{A}_{(j)}| \gtrsim Y/\tilde{H}$ for every $0 \le j \le k$,
\item $\mathcal{C}(\mathcal{A}_{(j)}) \ge \mathcal{C}(\mathcal{A}) \left( \frac{|\mathcal{A}_{(j)}|}{|\mathcal{A}|} \right)^{-1/2} - O \left( j \frac{|\mathcal{A}|^{1/2}}{|\mathcal{A}_{(j)}|^{1/2}}(\log Y)^{-B/2} \right)$, for every $0 \le j \le k$,
\item $\mathcal{C} (\mathcal{A}_{(j)}^{\uparrow p_{j+1}^{\ast}} ) \ge \mathcal{C}(\mathcal{A}_{(j)}) - O((\log Y)^{-B/2}) \gtrsim 1$, for every $0 \le j < k$.
\item there exists some $\tilde{c} \sim 1$, which can be made arbitrarily small provided $K$ is sufficiently large, such that given $1 \le j \le k$ and $(y,\alpha_y) \in \mathcal{A}_{(j)}$, there are $\gtrsim P^{1-\tilde{c}}$ pairs $( (x,\alpha_x),q) \in \mathcal{A}_{(j-1)} \times \mathcal{P}$ with $|y/q-x| \lesssim \tilde{H} \prod_{i=1}^{j-1} p_i^{\ast}$ and $\| q \alpha_y - \alpha_x \| \lesssim ( \tilde{H} \prod_{i=1}^{j-1} p_i^{\ast} )^{-1}$.
\end{enumerate}
\end{lema}

\begin{proof}
We start with $\mathcal{A}$ and apply Lemma \ref{mau1} to obtain a set $\mathcal{B} \subseteq \mathcal{A}$, $|\mathcal{B}| \gtrsim Y/\tilde{H}$, which we relabel as $\mathcal{A}_{(0)}$ and a prime $p^{\ast}$ which we label as $p_1^{\ast}$. Observe that from item (ii) of that lemma we know that we can find a subset ${\mathcal{B}}_{(1)} \subseteq \mathcal{A}_{(0)}^{\uparrow p_1^{\ast}}$ with $|{\mathcal{B}}_{(1)}| \ge |\mathcal{A}_{(0)}^{\uparrow p_1^{\ast}}| - O(\frac{1}{(\log Y)^{B/2}} \frac{Y}{\tilde{H}})$ and therefore $\mathcal{C}(\mathcal{B}_{(1)}) \ge \mathcal{C}(\mathcal{A}_{(0)}^{\uparrow p_1^{\ast}}) - O( (\log Y)^{-B/2})$, such that for every $(y,\alpha_y) \in {\mathcal{B}}_{(1)}$ there are $\gtrsim P^{1-\tilde{c}}$ pairs in $\mathcal{A}_{(0)} \times \mathcal{P}$ satisfying the conclusion of (v). Furthermore, from Lemma \ref{mau1} (i) we also have 
\begin{align*}
\mathcal{C} (\mathcal{A}_{(0)}^{\uparrow p_1^{\ast}} ) &\ge \mathcal{C}(\mathcal{A}_{(0)}) - O((\log Y)^{-B/2})) \\
&\ge \mathcal{C}(\mathcal{A}) \left( \frac{|\mathcal{A}_{(0)}|}{|\mathcal{A}|} \right)^{-1/2} - O((\log Y)^{-B/2}) \gtrsim 1,
\end{align*}
since $\mathcal{C}(A) \sim 1$. We now iterate this procedure. Once we know $|\mathcal{A}_{(j-1)}^{\uparrow p_j^{\ast}}| \gtrsim \frac{Y}{\tilde{H}} = \frac{\prod_{i=1}^j p_i^{\ast}   Y}{\prod_{i=1}^j p_i^{\ast}  \tilde{H}} $, $\mathcal{C}(\mathcal{A}_{(j-1)}^{\uparrow p_j^{\ast}}) \gtrsim 1$ and that we can find a subset $\mathcal{B}_{(j)} \subseteq \mathcal{A}_{(j-1)}^{\uparrow p_j^{\ast}}$ with 
$$|\mathcal{B}_{(j)}| \ge |\mathcal{A}_{(j-1)}^{\uparrow p_j^{\ast}}| - O(\frac{1}{(\log Y)^{B/2}} \frac{Y}{\tilde{H}}) \gtrsim Y/\tilde{H},$$ 
and therefore 
$$\mathcal{C}(\mathcal{B}_{(j)}) \ge \mathcal{C}(\mathcal{A}_{(j-1)}^{\uparrow p_j^{\ast}}) - O(\frac{1}{(\log Y)^{B/2}}) \gtrsim 1,$$
such that for every $(y,\alpha_y) \in {\mathcal{B}}_{(j)}$ there are $\gtrsim P^{1-\tilde{c}}$ pairs in $\mathcal{A}_{(j-1)} \times \mathcal{P}$ satisfying the conclusion of (v), we can then apply Lemma \ref{mau1} to obtain a subset $\mathcal{A}_{(j)} \subseteq \mathcal{B}_{(j)}$ with $|\mathcal{A}_{(j)}| \gtrsim Y/\tilde{H}$, a prime $p_{j+1}^{\ast}$ and a lifting $\mathcal{A}_{(j)}^{\uparrow p_{j+1}^{\ast}}$. Furthermore, this lemma tells us that we can find a corresponding subset $\mathcal{B}_{(j+1)}$ of $\mathcal{A}_{(j)}^{\uparrow p_{j+1}^{\ast}}$ that satisfies the analogues of the previous properties and that we have the estimates 
$$ \mathcal{C} (\mathcal{A}_{(j)}^{\uparrow p_{j+1}^{\ast}} ) \ge \mathcal{C}(\mathcal{A}_{(j)}) - O((\log Y)^{-B/2}),$$
and
\begin{align*}
\mathcal{C} (\mathcal{A}_{(j)}) &\ge \mathcal{C} (\mathcal{B}_{(j)}) |\mathcal{A}_{(j)}|^{-1/2} |\mathcal{B}_{(j)}|^{1/2} \\
&\ge \mathcal{C} (\mathcal{A}_{(j-1)}^{\uparrow p_j^{\ast}}) |\mathcal{A}_{(j)}|^{-1/2} |\mathcal{A}_{(j-1)}^{\uparrow p_j^{\ast}}|^{1/2} -  O((\log Y)^{-B/2})\\
&\ge \mathcal{C} (\mathcal{A}_{(j-1)}) \left( \frac{|\mathcal{A}_{(j)}|}{|\mathcal{A}_{(j-1)}|} \right)^{-1/2} - O((\log Y)^{-B/2}) \\
&\ge \mathcal{C} (\mathcal{A}) \left( \frac{|\mathcal{A}_{(j)}|}{|\mathcal{A}|} \right)^{-1/2} - O(j\frac{|\mathcal{A}|^{1/2}}{|\mathcal{A}_{(j)}|^{1/2}}(\log Y)^{-B/2}) \gtrsim 1,
\end{align*}
by induction. Notice that during the proof we used that we have the uniform bound (ii) for $\mathcal{A}_{(j-1)}$ to obtain that $|A_{(j)}| \gtrsim Y/ \tilde{H}$. By Lemma \ref{C1}, we then see that $\mathcal{C} (\mathcal{A}_{(j)})$ is bounded from above by an absolute constant. Since we also have $C(\mathcal{A}) \sim 1$ and $\mathcal{C} (\mathcal{A}_{(j)}) \ge \frac{1}{2} \mathcal{C} (\mathcal{A}) \left( \frac{|\mathcal{A}_{(j)}|}{|\mathcal{A}|} \right)^{-1/2}$, say, this ensures us that the bound (ii) remains uniform and this in turn guarantees the same claim for the remaining implicit constants in the statement. The result then follows.
\end{proof}

For technical reasons, we will need to ensure that for at least some of the configurations we construct, Lemma \ref{mau1} (ii) holds for $\gtrsim |\mathcal{P}|$ pairs instead of just $\gtrsim P^{1-\tilde{c}}$ pairs. This will be done rather crudely through the following lemma. 

\begin{lema}
\label{instead2}
Let $\mathcal{A}^{\uparrow p^{\ast}} [0]$ consist of those elements of $\mathcal{A}^{\uparrow p^{\ast}}$ corresponding to $\mathcal{A}_0$. Then, there are $\gtrsim |\mathcal{P}|^2 \frac{Y}{\tilde{H}}$ quadruples $((p^{\ast} x, \alpha_x^{\uparrow} ) , (p^{\ast} y, \alpha_y^{\uparrow}),q_1,q_2) \in \mathcal{A}^{\uparrow p^{\ast}} [0]^2 \times \mathcal{P}^2$ with $| p^{\ast} x / q_1 - p^{\ast} y / q_2 | < \frac{\delta \tilde{H} p^{\ast}}{P}$ and $\| q_1 \alpha_x^{\uparrow} - q_2 \alpha_y^{\uparrow} \| < \frac{C_2 P}{p^{\ast} \tilde{H}}$.
\end{lema}

\begin{proof}
This follows from items (ii) and (iv) of Proposition \ref{lift28}, the triangle inequality and the pigeonhole principle.
\end{proof}

We will now use these results to build a series of configurations starting from the configuration $\mathcal{J}$ provided by Lemma \ref{base00}.

\begin{prop}
\label{sets2}
Let the hypotheses and notation be as in Lemma \ref{base00} and let $1 \le i_1 \le i_2 \le k$ be integers with $k \lesssim \log X$ and $|i_2 - i_1| = O(1)$. Then, there exist primes $p_1^{\ast},\ldots,p_{k}^{\ast} \in \mathcal{P}$ and  $(c, H \prod_{i=1}^j p_i^{\ast})$-configurations $\mathcal{A}^{(j)}$ in $[X \prod_{i=1}^j p_i^{\ast}, 2X \prod_{i=1}^j p_i^{\ast}]$, for $0 \le j \le k$ and some $c \sim 1$, such that:
\begin{enumerate}[(i)]
\item $\mathcal{A}^{(0)} \subseteq \mathcal{J}$,
\item given $0 \le j_1 \le j_2 \le k$ and $(y,\alpha_y) \in \mathcal{A}^{(j_2)}$, there exists a unique $(x,\alpha_x) \in \mathcal{A}^{(j_1)}$ with $y = x \prod_{i=j_1+1}^{j_2} p_i^{\ast}$ and $\left(\prod_{i=j_1+1}^{j_2} p_i^{\ast} \right) \alpha_y = \alpha_x$,
\item there exists some $\tilde{c} \sim 1$, which can be made arbitrarily small provided $K$ is sufficiently large, such that given $1 \le j \le k$ and $(y,\alpha_y) \in \mathcal{A}^{(j)}$, there are $\gtrsim P^{1-\tilde{c}}$ pairs $( (x,\alpha_x),q) \in \mathcal{A}^{(j-1)} \times \mathcal{P}$ with $|y/q-x| \lesssim {H} \prod_{i=1}^{j-1} p_i^{\ast}$ and $\| q \alpha_y - \alpha_x \| \lesssim ( {H} \prod_{i=1}^{j-1} p_i^{\ast} )^{-1}$,
\item if $i_1 < j \le i_2$, then for every $(y,\alpha_y) \in \mathcal{A}^{(j)}$ we have $\gtrsim |\mathcal{P}|$ pairs $( (x,\alpha_x),q) \in \mathcal{A}^{(j-1)} \times \mathcal{P}$ of the above form.
\end{enumerate}
\end{prop}

\begin{proof}
In the range $0 \le j \le i_1$ the result follows from Lemma \ref{uni2}, with $\mathcal{A}^{(j)}=\mathcal{A}_{(j)}$. Notice that by construction of the liftings $\mathcal{A}_{(j-1)}^{\uparrow p_{j}^{\ast}}$ and the estimates on $\mathcal{C}(\mathcal{A}_{(j)})$, we can ensure a uniform bound throughout the process for the parameters $c_1, c_2, C_2$ and $\delta$ in Proposition \ref{lift28}.  In the range $(i_1,i_2]$ we use Lemma \ref{instead2} instead, with $\mathcal{A}^{(j)}$ then given, using the notation in the statement of that lemma, by $(\mathcal{A}^{(j-1)})^{\uparrow p_j^{\ast}} [0]$. Since the range $(i_1,i_2]$ involves $O(1)$ elements, we still end with a configuration $\mathcal{A}^{(i_2)}$ that satisfies the hypotheses of Proposition \ref{lift28} for certain parameters $c_1, c_2, C_2, \delta \sim 1$. We can then apply Lemma \ref{uni2} again to cover the remaining range $(i_2,k]$ and conclude the proof.
\end{proof}

\section{Downward expansion}

We will now develop analogues of some estimates from \cite{W2} that will allow us to study certain 'downward paths' associated to the configurations we have constructed. We will use this in the next section to show that, if $k$ is sufficiently large, every fixed choice of $(x_0, \alpha_{x_0}) \in \mathcal{A}^{(k)}$ is connected through products of primes in $\mathcal{P}$ to $\gtrsim X/H$ elements of $\mathcal{J}$.

\begin{defi}
Let the hypotheses and notation be as in Proposition \ref{sets2}. A \emph{downward path} of length $s$ is a set of elements $(x_0,\alpha_{x_0}), (x_1, \alpha_{x_1}),\ldots,(x_s,\alpha_{x_s})$ with $(x_i, \alpha_{x_i}) \in \mathcal{A}^{(j-i)}$ for every $0 \le i \le s$ and some $1 \le j \le k$, such that for every $0 \le i \le s-1$ there exists a prime $q_i \in \mathcal{P}$ with $|x_i/q_i - x_{i+1}| \lesssim {H}\prod_{t=1}^{j-i-1} p_t^{\ast}$ and $\| q_i \alpha_{x_i} - \alpha_{x_{i+1}} \| \lesssim   \left( {H}\prod_{t=1}^{j-i-1} p_t^{\ast} \right)^{-1}$. We say the downward path \emph{connects} $(x_0,\alpha_{x_0})$ and $(x_s,\alpha_{x_s})$. Given $Q \in \N$, we say the downward path is coprime with $Q$ if no $q_i$ divides $Q$.
\end{defi}

The following two lemmas are variants of results from \cite[Section 4]{W2}.

\begin{lema}
\label{prod2}
There exists $c \sim 1$ such that, if $q_0,\ldots,q_{s-1}$ are the primes associated to a downward path of length $s \le c \log (X/{H})$ given by elements $(x_i,\alpha_{x_i}) \in \mathcal{A}^{(j-i)}$, $0 \le i \le s$, then $\prod_{i=0}^{s-1} q_i/p_{j-i}^{\ast} \sim 1$.
\end{lema}

\begin{proof}
Each $x_i$ may be written as $\tilde{x}_i \prod_{t=1}^{j-i} p_t^{\ast}$, for some $\tilde{x}_i$ which is the first coordinate of an element of $\mathcal{J}$. We then have that
$$ | \frac{\tilde{x}_i }{q_i}\prod_{t=1}^{j-i} p_t^{\ast} - \tilde{x}_{i+1}  \prod_{t=1}^{j-i-1} p_t^{\ast} | \lesssim {H}\prod_{t=1}^{j-i-1} p_t^{\ast},$$
for every $0 \le i < s$, which after simplifying implies that
$$ | \frac{p_{j-i}^{\ast}}{q_i} - \frac{\tilde{x}_{i+1}}{\tilde{x}_{i}} |  \lesssim \frac{{H}}{X}.$$
In particular, we see that
$$ \prod_{i=0}^{s-1} \frac{p_{j-i}^{\ast}}{q_i}  = \prod_{i=0}^{s-1} \left( \frac{\tilde{x}_{i+1}}{\tilde{x}_{i}} + O ({H}/X ) \right).$$
We expand this into $2^s$ terms, with the first one given by $\tilde{x}_{s}/\tilde{x}_0 \sim 1$. Since the remaining terms are all bounded by $O(2^s {H}/X)$, we obtain the result as long as $c \sim 1$ is sufficiently small in the statement.
\end{proof}

\begin{lema}
\label{down3}
Let the hypotheses be as in Lemma \ref{prod2}. Then, there exists $C \sim 1$ such that
\begin{enumerate}[(i)]
\item $|x_0/\prod_{i=0}^{s-1} q_i - x_s| \le C s {H} \prod_{t=1}^{j-s} p_t^{\ast}$,
\item $\| \left( \prod_{i=0}^{s-1} q_i \right) \alpha_{x_0} - \alpha_{x_s} \| \le C s  \left( {H} \prod_{t=1}^{j-s} p_t^{\ast} \right)^{-1}$.
\end{enumerate}
\end{lema}

\begin{proof}
In both cases we proceed by induction on $s$, the result following for $s=1$ from the definition of a downward path. Let then $s > 1$. To see (i), observe that by induction we have $|x_1 / \prod_{i=1}^{s-1} q_i - x_s| \le C (s-1) {H} \prod_{t=1}^{j-s} p_t^{\ast}$. On the other hand, from the definition of a downward path and Lemma \ref{prod2}, we have that $( \prod_{i=1}^{s-1} q_i)^{-1} | x_0/q_0 - x_1| \lesssim {H} \prod_{t=1}^{j-s} p_t^{\ast}$. Claim (i) then follows using the triangle inequality provided $C$ was chosen sufficiently large. Finally, claim (ii) follows from an entirely analogous argument using that ${H}$ is sufficiently large.
\end{proof}

We also have the following analogue of \cite[Lemma 5.1]{W2}.

\begin{lema}
\label{expansion3}
Suppose we are given two downward paths as in Lemma \ref{prod2}, $(x_0,\alpha_{x_0})$, $(x_1, \alpha_{x_1}) , \ldots , (x_s,\alpha_{x_s})$ and $(x_0',\alpha_{x_0'})$, $(x_1', \alpha_{x_1'}) , \ldots , (x_s',\alpha_{x_s'})$, consisting of primes $q_0, \ldots, q_{s-1}$ and $q_0',\ldots,q_{s-1}'$, respectively. Assume both paths share the same initial and end points, that is, $(x_0,\alpha_{x_0}) = (x_0',\alpha_{x_0'})$ and $(x_s,\alpha_{x_s}) = (x_s',\alpha_{x_s}')$. Then, either $\prod_{i=0}^{s-1} q_i = \prod_{i=0}^{s-1} q_i'$ or $ \prod_{t=j-s+1}^{j} p_t^{\ast}  \gtrsim \frac{X}{s {H}}$.
\end{lema}

\begin{proof}
It follows from Lemma \ref{down3} (i) that
$$ \left| (\prod_{i=0}^{s-1} q_i)^{-1} - (\prod_{i=0}^{s-1} q_i')^{-1} \right| \lesssim s \frac{{H} \prod_{t=1}^{j-s} p_t^{\ast}}{X \prod_{t=1}^{j} p_t^{\ast}}.$$
The result then follows simplifying the right-hand side, multiplying both sides by $\left(  \prod_{t=j-s+1}^{j} p_t^{\ast} \right)^2$ and applying Lemma \ref{prod2}.
\end{proof}

We finish this section showing how the mixing lemma implies certain expansion properties of downward paths.

\begin{lema}
\label{grow26}
Let the notation be as in Proposition \ref{sets2} and let $Q \in \N$ with $Q \le X$. Let $S$ be a subset of $\mathcal{A}^{(j)}$ for some $1 \le j \le \log(X/H)$ and write $S' \subseteq \mathcal{A}^{(j-1)}$ for those elements that are connected to an element of $S$ by a downward path coprime with $Q$. Let $N_j = 1$ if $i_1 < j \le i_2$ and $N_j = P^{-\tilde{c}}$ otherwise, with $\tilde{c}$ as in Proposition \ref{sets2} (iii). Then $|S'| \ge \min \left\{c N_j^2 X/{H}, |S| P^{c^{\ast}} \right\}$ for some $c, c^{\ast} \sim 1$, with $c^{\ast}$ an absolute constant.
\end{lema}

\begin{proof}
By items (iii) and (iv) of Proposition \ref{sets2}, we know that for each $(y,\alpha_y) \in S$ there exist $\gtrsim N_j|\mathcal{P}|$ pairs $((x,\alpha_x),q) \in S' \times \mathcal{P}$ with $(y,\alpha_y)$ connected to $(x,\alpha_x)$ through the prime $q$. Because of our size hypothesis on $Q$, as long as $K$ is sufficiently large we may assume all these primes $q$ are coprime with $Q$, after adjusting the implicit constant if necessary. By Cauchy-Schwarz, the number of triples $((x,\alpha_x),q_1,q_2) \in S' \times \mathcal{P}^2$ for which there exist $(y_1,\alpha_{y_1}),(y_2,\alpha_{y_2}) \in S$ connected with $(x,\alpha_x)$ through $q_1$ and $q_2$, respectively, is then at least $ \gtrsim \frac{\left( |S| |\mathcal{P}| N_j \right)^2}{|S'|}$. Since each such triple leads to a relation $(y_1,\alpha_{y_1}) \sim_{q_1,q_2} (y_2,\alpha_{y_2})$, it thus follows from Corollary \ref{ML0} that
$$ |S|^2 |\mathcal{P}|^2 N_j^2 \lesssim |S'| \left( \frac{{H}}{X} |S|^2 |\mathcal{P}|^2 + |S| P^{2-c_0} \right).$$
Since $c_0 > 0$ is an absolute constant, the result follows choosing $K$ sufficiently large as to make $\tilde{c}$ sufficiently small.
\end{proof}

\section{Building a global frequency}

We will now show how the estimates we have developed so far can be used to prove Theorem \ref{2}, albeit with a slightly worse bound of the form $|T| \lesssim (X^2/H^2) \log X$. In the next section, we will show how this additional $\log X$ factor can be removed to conclude the proof.

With this in mind, we let the hypotheses and notation be as in Proposition \ref{sets2}, for some choices of $i_1 < i_2 < i_3 < k$ that we are about to specify. We let $k-i_3= (1-\rho) \lceil \frac{\log X/{H}}{\log P} \rceil$, $i_3 - i_2 = \lceil \frac{\sigma}{c^{\ast}} \frac{\log X/{H}}{\log P} \rceil$, $|i_2 - i_1|$ equal to a sufficiently large constant and $i_1 = (1+\gamma) \lceil \frac{\log X/{H}}{\log P} \rceil$, where $0 < \rho < \sigma < \gamma < 1/100$, $c^{\ast}$ is as in Lemma \ref{grow26}, $\sigma$ is sufficiently small with respect to $\gamma$ and $\rho$ sufficiently small with respect to $\sigma$. We fix an element $(x_0,\alpha_{x_0}) \in \mathcal{A}^{(k)}$ and given $Q \in \N$, we write $\mathcal{A}_{x_0,Q}^{(j)}$ for those elements of $\mathcal{A}^{(j)}$ that are the endpoint of a downward path coprime with $Q$ of length $k-j$ with initial point $(x_0,\alpha_{x_0})$.  We begin with the following estimate.

\begin{lema}
\label{mon6}
Let $\tilde{c}$ be as in Lemma \ref{grow26}. Assume $K$ is sufficiently large with respect to $\sigma, \rho$ and let $Q \in \N$ with $Q \le X$. Then, $|\mathcal{A}_{x_0,Q}^{(i_3)}| \gtrsim \left( X/{H} \right)^{1-\sigma}$, $|\mathcal{A}_{x_0,Q}^{(i_2)}| \gtrsim P^{-2 \tilde{c}} X/{H}$ and $|\mathcal{A}_{x_0,Q}^{(j)}| \gtrsim X/{H}$ for every $j < i_2$.
\end{lema}

\begin{proof}
Since $k-i_3= (1-\rho) \lceil \frac{\log X/{H}}{\log P} \rceil$, we see from Lemma \ref{expansion3} that any pair of downward paths that start at $(x_0,\alpha_{x_0})$ and have the same endpoint in $\mathcal{A}_{x_0,Q}^{(i_3)}$ must consist of the same primes, counting repetitions. On the other hand, from Proposition \ref{sets2} (iii) we know there exist some $c, \tilde{c} \sim 1$ such that we have at least $ \left( c P^{1-\tilde{c}} \right)^{k-i_3}$ such downward paths, where furthermore $\tilde{c}$ can be made arbitrarily small provided $K$ is sufficiently large and we can take these paths to be coprime with $Q$ because of our size hypothesis on $Q$, upon adjusting $c$ if necessary. Combining these observations, the bound $|\mathcal{A}_{x_0,Q}^{(i_3)}| \gtrsim \left( X/{H} \right)^{1-\sigma}$ follows from Stirling's formula using that $K$ can be chosen sufficiently large with respect to $\sigma$. The bound $|\mathcal{A}_{x_0,Q}^{(i_2)}| \gtrsim P^{-2 \tilde{c}} X/{H}$ then follows from iterating Lemma \ref{grow26} starting with $S=\mathcal{A}_{x_0,Q}^{(i_3)}$, using that $i_3 - i_2 = \lceil \frac{\sigma}{c^{\ast}} \frac{\log X/{H}}{\log P} \rceil$. Similarly, applying Lemma \ref{grow26} with $S=\mathcal{A}_{x_0,Q}^{(i_2)}$ and using that $K$ may be chosen sufficiently large as to make $\tilde{c}$ much smaller than the constant $c^{\ast}$ in that lemma, we obtain that $|\mathcal{A}_{x_0,Q}^{(i_2-1)}| \gtrsim X/{H}$. Since we know from Proposition \ref{sets2} (ii) that $|\mathcal{A}^{(j_1)}| \ge |\mathcal{A}^{(j_2)}|$ whenever $0 \le j_1 \le j_2 \le k$, this concludes the proof.
\end{proof}

\begin{lema}
\label{d6}
Let $\tau > 0$ and assume $\sigma$ is sufficiently small with respect to $\tau$. Then, there exist coprime integers $a,b \lesssim X^{\tau}$ such that $\alpha_{x_0} = \frac{a}{b} + O\left( \frac{kX}{{H}^2 \prod_{t=1}^k p_t^{\ast}} \right)$.
\end{lema}

\begin{proof}
Let $i_1 < j < i_2$. By Lemma \ref{mon6}, we know that $|\mathcal{A}_{x_0,1}^{(j)}| \gtrsim X/{H}$. To each $(x,\alpha_x) \in \mathcal{A}_{x_0,1}^{(j)}$ we may associate a sequence of primes $q_0^x,\ldots,q_{s-1}^x$, $s=k-j$, for which there exists a downward path of length $s$ consisting of these primes and connecting $(x_0,\alpha_{x_0})$ with $(x,\alpha_x)$. By Lemma \ref{down3} (i), we can find a subset $\mathcal{B} \subseteq \mathcal{A}_{x_0,1}^{(j)}$ of size $|\mathcal{B}| \gtrsim \frac{X}{s {H}}$ such that the points $x_0/ \prod_{i=1}^{s-1} q_i^x$ corresponding to elements $(x,\alpha_x) \in \mathcal{B}$ form an ${H} \prod_{t=1}^j p_t^{\ast}$-separated set $B$ of points in $[X \prod_{t=1}^j p_t^{\ast}, 2 X \prod_{t=1}^j p_t^{\ast}]$. By Proposition \ref{sets2} (iv) and the Cauchy-Schwarz and triangle inequalities, we know there are $ \gtrsim \frac{H}{X} |B|^2 |\mathcal{P}|^2$ quadruples $( (x,\alpha_x),(y,\alpha_y) , q_s^x, q_s^y ) \in \mathcal{B}^2 \times \mathcal{P}^2$ with $| x/q_s^x - y/q_s^y| \lesssim {H} \prod_{t=1}^{j-1} p_t^{\ast}$ and $\| q_s^x \alpha_x - q_s^y \alpha_y \| \lesssim \left( {H} \prod_{t=1}^{j-1} p_t^{\ast} \right)^{-1}$. By Corollary \ref{ML0} applied with $\mathcal{A}_1=\mathcal{A}_2=B$ and $\tilde{H} = \epsilon {H} \prod_{t=1}^j p_t^{\ast}$ for some sufficiently small $\epsilon \sim 1$, we then deduce that there must exist at least one such quadruple $( (x,\alpha_x),(y,\alpha_y) , q_s^x, q_s^x ) \in \mathcal{B}^2 \times \mathcal{P}^2$ such that the corresponding primes satisfy
$$ \left| \frac{x_0}{\prod_{i=0}^s q_i^x} -   \frac{x_0}{\prod_{i=0}^s q_i^y} \right| \gtrsim {H} \prod_{t=1}^{j-1} p_t^{\ast}.$$
Rearranging and using Lemma \ref{prod2}, this tells us that
$$ \left| \prod_{i=0}^s q_i^x - \prod_{i=0}^s q_i^y \right| \gtrsim \frac{{H}}{X} \prod_{t=j}^k p_t^{\ast}.$$
By construction of these primes and estimates (i) and (ii) from Lemma \ref{down3}, we also have that
$$ \left| \prod_{i=0}^s q_i^x - \prod_{i=0}^s q_i^y \right| \lesssim s \frac{{H}}{X} \prod_{t=j}^k p_t^{\ast} \lesssim X^{\tau},$$
provided $\sigma$ is sufficiently small with respect to $\tau$, and
$$ \left\| \left(  \prod_{i=0}^s q_i^x \right) \alpha_{x_0} - \left(  \prod_{i=0}^s q_i^y \right) \alpha_{x_0} \right\| \lesssim s \left( {H} \prod_{t=1}^{j-1} p_t^{\ast} \right)^{-1}.$$
It follows that
$$ \alpha_{x_0} = \frac{a_0}{b_0} + O \left( \frac{s}{b_0  {H} \prod_{t=1}^{j-1} p_t^{\ast}} \right),$$
with $b_0 = \left| \prod_{i=0}^s q_i^x - \prod_{i=0}^s q_i^y \right|$ and $a_0 \in \Z$, which gives the result.
\end{proof}

We will now show that it must be $b=O(1)$. This lemma is the only place in the proof where use the full generality of Lemma \ref{ML} and therefore the Riemann Hypothesis for Dirichlet L-functions.

\begin{lema}
\label{b6}
Let the hypotheses be as in Lemma \ref{d6} and assume $\tau$ is sufficiently small with respect to $\gamma$. Then $b=O(1)$.
\end{lema}

\begin{proof}
Let $i_1 < j < i_2$. By Lemma \ref{prod2}, Lemma \ref{down3} (ii) and Lemma \ref{d6}, we see that for every $(x,\alpha_x) \in \mathcal{A}_{x_0,b}^{(j)}$ there exists some integer $a_x$ coprime with $b$ such that
\begin{equation}
\label{Op6}
 \| \alpha_x - \frac{a_x}{b} \| \lesssim \frac{kX}{{H}^2 \prod_{t=1}^j p_t^{\ast}} < X^{-2\tau} ,
 \end{equation}
say, provided $\tau$ is sufficiently small with respect to $\gamma$. Furthermore, since $|\mathcal{A}_{x_0, b}^{(j)}| \gtrsim X/{H}$, we see from Proposition \ref{sets2} (iv) and the Cauchy-Schwarz and triangle inequalities that there are $\gtrsim \frac{X}{{H}} |\mathcal{P}|^2$ quadruples $( (x,\alpha_x), (y,\alpha_y), p, q)$ in $(\mathcal{A}_{x_0, b}^{(j)})^2 \times \mathcal{P}^2$, with $\| p \alpha_x - q \alpha_y \| \lesssim ( {H} \prod_{t=1}^{j-1} p_t^{\ast} )^{-1}$ and $|x/p - y/q| \lesssim {H} \prod_{t=1}^{j-1} p_t^{\ast}$, where adjusting the implicit constant if necessary we may assume they all involve primes coprime with $b$. Using (\ref{Op6}), we see that each such quadruple must satisfy $p a_x \equiv q a_y \, (\text{mod }b)$.

Let $s = M \lceil \frac{\log b}{\log P} \rceil$ for some sufficiently large $M \in \N$. By an application of the Blakley-Roy inequality \cite{BR}, we deduce from our last observation that the number of $(2s+2)$-tuples of elements $((x,\alpha_x), (y, \alpha_y),p_1,\ldots,p_s,q_1,\ldots,q_s) \in (\mathcal{A}_{x_0,b}^{(j)})^2 \times \mathcal{P}^{2s}$, consisting of primes coprime with $b$, with $a_x a_y^{-1} \equiv \prod_{i=1}^s q_i p_i^{-1} \, (\text{mod }b)$ and with $(x,\alpha_x)$ connected to $(y,\alpha_y)$ through a path of length $s$ (Definition \ref{path}) involving these primes, is at least $\ge \frac{X}{{H}} c^s |\mathcal{P}|^{2s}$ for some $c \sim 1$, where the inverses are taken mod $b$. We will now proceed to assume $b$ is sufficiently large and use Lemma \ref{ML} to derive a contradiction. 

Let us abbreviate $\mathcal{B} := \mathcal{A}_{x_0,b}^{(j)}$ and given an element $u \in (\Z / b \Z)^*$, let us write $\mathcal{B}[u]$ for those elements $(x,\alpha_x) \in \mathcal{B}$ with $a_x \equiv u \, (\text{mod }b)$. It is clear that we can find some $u \in (\Z / b \Z)^*$ such that at least $\gtrsim \max \left\{ | \mathcal{B}[u] |, \frac{X}{b{H}} \right\} c^s |\mathcal{P}|^{2s}$ of the $(2s+2)$-tuples have their initial point $(x,\alpha_x)$ in $\mathcal{B}[u]$. Let us fix such a choice of $u$.

Let now $\epsilon \sim 1$ be a sufficiently small absolute constant and assume at least half of these paths with initial point in $\mathcal{B}[u]$ have their end points in $\bigcup_v^{'} \mathcal{B}[v]$, where the union is restricted to those $v \in (\Z / b \Z)^*$ such that $|\mathcal{B}[v]| > b^{-\epsilon} \frac{X}{{H}}$. If this is the case, we can find some $v$ such that the number of such paths with end point in $\mathcal{B}[v]$ is at least $\gtrsim b^{-\epsilon} \max \left\{ | \mathcal{B}[u] |, \frac{X}{b{H}} \right\} c^s |\mathcal{P}|^{2s}$. We can then apply Lemma \ref{ML} with $k=s$, $Q = b$ and $a= u v^{-1}$ to conclude that
\begin{align*}
b^{-\epsilon} \max \left\{ | \mathcal{B}[u] |, \frac{X}{b{H}} \right\}  c^s & |\mathcal{P}|^{2s} \lesssim \frac{{H}}{X} \frac{C^{s}}{\varphi(b)} |\mathcal{P}|^{2s} |\mathcal{B}[u]| |\mathcal{B}[v]| \\
&+ C^{s} |\mathcal{B}[u]|^{1/2} |\mathcal{B}[v]|^{1/2} |\mathcal{P}|^{2s}  \varphi (b) P^{-c_0 s},
 \end{align*}
for some $c_0,C \sim 1$. Using that $\varphi(b) \gtrsim_{\mu} b^{1-\mu}$ for any $\mu > 0$, this can easily be seen to give a contradiction if $b$ and $M$ are sufficiently large, by using the first member of the $\max$ for the first term in the upper bound and the second member of the $\max$ for the second term in the upper bound.

We may therefore assume that there is some $v \in (\Z / b \Z)^*$ with $|\mathcal{B}[v]| \le b^{-\epsilon} \frac{X}{H}$ such that $\mathcal{B}[v]$ is the endpoint of at least $\gtrsim \frac{1}{b} \max \left\{ | \mathcal{B}[u] |, \frac{X}{b{H}} \right\} c^s |\mathcal{P}|^{2s}$ of the paths with initial point in $\mathcal{B}[u]$. Applying Lemma \ref{ML} as before, we now obtain
\begin{align*}
\frac{1}{b} \max \left\{ | \mathcal{B}[u] |, \frac{X}{b{H}} \right\}  c^s & |\mathcal{P}|^{2s} \lesssim b^{-\epsilon} \frac{C^{s}}{\varphi(b)} |\mathcal{P}|^{2s} |\mathcal{B}[u]| \\
&+ C^{s} |\mathcal{B}[u]|^{1/2} |\mathcal{B}[v]|^{1/2} |\mathcal{P}|^{2s}  \varphi (b) P^{-c_0 s},
 \end{align*}
 from where we can again obtain a contradiction upon choosing $M$ sufficiently large with respect to the implicit constants and $b$ sufficiently large with respect to the implicit constants and $M$.
 \end{proof}

We can now 'almost' prove the result we are after.

\begin{prop}
\label{almost6}
Let the hypotheses and notation be as in Lemma \ref{base00}. Then, there exists $T_0 \in \R$ with $|T_0| \lesssim (X^2/H^2) \log X$, $b \in \N$ with $b=O(1)$ and a subset $\tilde{\mathcal{J}}  \subseteq \mathcal{J}$ with $| \tilde{\mathcal{J}}| \gtrsim X/H$ such that, for every $(x,\alpha_x) \in \tilde{\mathcal{J}}$, there exists some $a_x \in \N$ with
\begin{equation}
\label{T7}
 \alpha_x \equiv \frac{a_x}{b} + \frac{T_0}{x}  + O(\frac{\log^2 X}{H}) \, \, \, (\text{mod }1).
 \end{equation}
 Furthermore, there are $\gtrsim \frac{X}{H} |\mathcal{P}|^2$ pairs $(x,\alpha_x), (y, \alpha_y) \in \tilde{\mathcal{J}}$ with $(x, \alpha_x) \sim (y,\alpha_y)$.
\end{prop}

\begin{proof}
By Proposition \ref{sets2} (ii) we know there exists some $x_0' \in [X,2X]$ with $x_0 = x_0' \prod_{t=1}^k p_t^{\ast}$. Since $k \lesssim \log X$, it follows from Lemma \ref{d6} and Lemma \ref{b6} that there exists some $T_0 \in \R$ with $|T_0| \lesssim (X^2/H^2) \log X$ such that $ \alpha_{x_0} = \frac{a}{b} + \frac{T_0}{x_0}$, for some $a,b \in \N$ of size $O(1)$. It then follows from Lemma \ref{down3} and Lemma \ref{mon6} that we can find some $i_1 < j < i_2$ such that $|\mathcal{A}_{x_0,1}^{(j)}| \gtrsim X/H$ and such that, for every $(x,\alpha_x) \in \mathcal{A}_{x_0,1}^{(j)}$, there exists some $a_x \in \N$ and primes $q_0,\ldots,q_{k-j-1}$ with 
\begin{equation}
\label{dist6}
|x_0 / \prod_{i=0}^{k-j-1} q_i  - x| \lesssim k {H} \prod_{t=1}^j p_t^{\ast} ,
\end{equation}
and
\begin{equation}
\label{aprox6}
 \alpha_x \equiv \frac{a_x}{b} + T_0 \frac{\prod_{i=0}^{k-j-1} q_i}{x_0} + O \left( \frac{k}{ {H} \prod_{t=1}^j p_t^{\ast}} \right) \, \, (\text{mod }1).
 \end{equation}
Since again by Proposition \ref{sets2} (ii) we have some $x' \in [X,2X]$ with $x=x' \prod_{t=1}^j p_t^{\ast}$, we deduce from (\ref{dist6}) that
$$ |T_0| \left| \frac{\prod_{i=0}^{k-j-1} q_i}{x_0} - \frac{1}{x} \right| \lesssim |T_0| \frac{k H}{X^2} \frac{ \left( \prod_{t=1}^j p_t^{\ast} \right) \left( \prod_{i=0}^{k-j-1} q_i \right)}{ \left( \prod_{t=1}^k p_t^{\ast} \right) \left( \prod_{t=1}^j p_t^{\ast} \right)}.$$
Using Lemma \ref{prod2}, it then follows from (\ref{aprox6}) that every element $(x,\alpha_x) \in \mathcal{A}_{x_0,1}^{(j)}$ satisfies
\begin{equation}
\label{f6}
 \alpha_x \equiv \frac{a_x}{b} + \frac{T_0}{x} + O \left( \frac{\log^2 X}{H  \prod_{t=1}^j p_t^{\ast}} \right) \, \, (\text{mod }1).
 \end{equation}
But as we have just noticed, for each such $(x,\alpha_x)$ there exists some $(x',\alpha_{x'}) \in \mathcal{A}^{(0)} \subseteq \mathcal{J}$ with $x=x' \prod_{t=1}^j p_t^{\ast}$ and $ \left( \prod_{t=1}^j p_t^{\ast}  \right) \alpha_x = \alpha_{x'}$. Estimate (\ref{T7}) for $(x',\alpha_{x'})$ then follows immediately from (\ref{f6}). Finally, we also know from Proposition \ref{sets2} (iv) and the Cauchy-Schwarz and triangle inequalities, that there are $\gtrsim \frac{X}{H} |\mathcal{P}|^2$ pairs $(x,\alpha_x), (y,\alpha_y) \in \mathcal{A}_{x_0,1}^{(j)}$ with $|x/p - y /q| \lesssim {H} \prod_{t=1}^{j-1} p_t^{\ast}$ and $\| p \alpha_x - q \alpha_y \| \lesssim ({H} \prod_{t=1}^{j-1} p_t^{\ast})^{-1}$, for some $p,q \in \mathcal{P}$, which similarly implies the last claim in the statement.
\end{proof}

\section{Sharpening}

We will now conclude the proof of Theorem \ref{2}, and therefore Theorem \ref{1}, by showing how to remove the extra logarithmic factors in Proposition \ref{almost6}. Throughout this section we will let the hypotheses and notation be as in that proposition. We begin with the following estimate.

\begin{lema}
\label{T9}
We have $|T_0| \lesssim X^2/H^2$.
\end{lema}

\begin{proof}
Let us assume that $|T_0| \gtrsim X^2/H^2$ for a sufficiently large implicit constant and derive a contradiction. Since given $(x,\alpha_x) \in \tilde{\mathcal{J}}$ and $\tilde{x} \in \R$, we have
$$ \left| \frac{T_0}{x} - \frac{T_0}{\tilde{x}} \right| = |T_0| \frac{|x-\tilde{x}|}{|x \tilde{x}|},$$
it follows from (\ref{T7}) that we can find some $\tilde{x}$ with $|x-\tilde{x}| \lesssim H \log^2 X$ and
$$ \alpha_x \equiv \frac{a_x}{b} + \frac{T_0}{\tilde{x}} \, \, (\text{mod }1).$$
Let us associate such an element $\tilde{x}$ to each $(x,\alpha_x) \in \tilde{\mathcal{J}}$. Suppose now that $(x,\alpha_x), (y,\beta_y) \in \tilde{\mathcal{J}}$ are such that $(x,\alpha_x) \sim_{p,q} (y,\beta_y)$ for some $p,q \in \mathcal{P}$. It follows that 
$$ \frac{T_0}{\tilde{x}} p \equiv \frac{T_0}{\tilde{y}}q + \frac{a_{x,y}}{b}+ O(P/H) \, \, (\text{mod }1),$$
for some integer $a_{x,y} $. Rearranging, this means that
\begin{equation}
\label{bf20}
 T_0 \frac{p\tilde{y} - q \tilde{x}}{\tilde{x} \tilde{y}} \equiv \frac{a_{x,y}}{b} + O(P/H) \, \, (\text{mod 1}).
 \end{equation}
Since $|qx - py| \lesssim PH$, we see from the definition of the elements $\tilde{x}, \tilde{y}$ and the triangle inequality that the real number on the left hand side is $ \lesssim (P/H) \log^3 X$. Given that $b=O(1)$, (\ref{bf20}) then necessary implies that it must in fact be
\begin{equation}
\label{pm9}
 \left| T_0 \frac{p\tilde{y} - q \tilde{x}}{\tilde{x} \tilde{y}} \right| \lesssim P/H.
 \end{equation}
It will thus suffice to find a pair $(x,\alpha_x), (y,\beta_y) \in \tilde{\mathcal{J}}$ with $(x,\alpha_x) \sim_{p,q} (y,\beta_y)$ and $|\tilde{x}/p-\tilde{y}/q| \ge \epsilon H/P$ for some $\epsilon \sim 1$. Let us write $B$ for the set of elements $\tilde{x}$ arising from the elements $x \in \tilde{\mathcal{J}}$. Since we know that there are $\gtrsim \frac{X}{H} |\mathcal{P}|^2$ pairs $(x,\alpha_x) \sim (y,\beta_y)$, if no pair of the desired form exists, we can find a subset $B_1 \subseteq B$ with $|B_1| \sim \frac{|B|}{\log^2 X}$ consisting of $H$-separated elements and such that for at least $\gtrsim \frac{X}{H} \frac{|\mathcal{P}|^2}{\log^2 X}$ of the pairs $(x,\alpha_x) \sim_{p,q} (y,\beta_y)$ we have $\tilde{x} \in B_1$ and $|\tilde{x}/p-\tilde{y}/q| < \epsilon H/P$. This in turn allows us to find a set $B_2 \subseteq B$ of $H$-separated points with $|B_2 | \sim \frac{|B|}{\log^2 X}$ and such that at least $\gtrsim \frac{X}{H} \frac{|\mathcal{P}|^2}{\log^4 X}$ of the pairs $(x,\alpha_x) \sim_{p,q} (y,\beta_y)$ satisfy $\tilde{x} \in B_1$, $\tilde{y} \in B_2$ and $|\tilde{x}/p-\tilde{y}/q| < \epsilon H/P$. Since this can easily be seen to contradict Corollary \ref{ML0} if $\epsilon \sim 1$ is sufficiently small and $K$ is sufficiently large, the result follows.
\end{proof}

Given $(x,\alpha_x) \in \tilde{\mathcal{J}}$, we know from Proposition \ref{almost6} that we can write
\begin{equation}
\label{w21}
  \alpha_x \equiv \frac{a_x}{b} + \frac{T_0+T_x}{x}  \, \, \, (\text{mod }1),
  \end{equation}
for some real number $T_x$ with $|T_x| \lesssim (X/H) \log^2 X$. If  $(x,\alpha_x) \sim_{p,q} (y,\beta_y)$, we have that
$$ \frac{p T_x}{x} - \frac{q T_y}{y} \equiv \frac{a_{x,y}}{b} + \frac{T_0 (q x - p y)}{xy} + O(P/H) \equiv \frac{a_{x,y}}{b} + O(P/H) \, \, \, (\text{mod }1)$$
for some integer $a_{x,y}$ and where we have used Lemma \ref{T9}. Given the size of $T_x, T_y$ and that $b=O(1)$, this implies that $| \frac{T_x}{x} - \frac{q T_y}{p y}| \lesssim 1/H$ as real numbers. But since $| \frac{q T_y}{p y} - \frac{T_y}{x} | \lesssim \frac{\log^2 X}{X} \lesssim 1/H$, we deduce from the triangle inequality that
\begin{equation}
\label{tilde21}
 |T_x - T_y | \lesssim X/H.
 \end{equation}
Recall that $|T_x| \lesssim (X/H) \log^2 X$ for every $(x,\alpha_x) \in \tilde{\mathcal{J}}$. We now cover this range with disjoint intervals of the form $I_r = (r,r + CX/H]$ for some $C \sim 1$ and write $\tilde{\mathcal{J}}[r]$ for those $(x,\alpha_x) \in \tilde{\mathcal{J}}$ with $T_x \in I_r$. By (\ref{tilde21}) and an easy pigeonholing argument, we see that if $C \sim 1$ is sufficiently large, this covering can be done in such a way that for $\gtrsim \frac{X}{H} |\mathcal{P}|^2$ of the pairs $(x,\alpha_x), (y,\beta_y) \in \tilde{\mathcal{J}}$ with $(x,\alpha_x) \sim (y,\beta_y)$ there exists some $r$ such that both elements lie in the same subset $\tilde{\mathcal{J}}[r]$. In particular, this guarantees that there is some $r_0$ such that $\gtrsim |\tilde{\mathcal{J}}[r_0]| |\mathcal{P}|^2$ of these pairs consist of elements of $\tilde{\mathcal{J}}[r_0]$. By Corollary \ref{ML0}, this implies that it must be $|\tilde{\mathcal{J}}[r_0]| \gtrsim X/H$. Writing $T = T_0+r_0$ and using (\ref{w21}), this concludes the proof of Theorem \ref{2}.


\begin{thebibliography}{10}
\bibitem{BR}  G. Blakley, P. Roy. \emph{A Hölder type inequality for symmetric matrices with nonnegative entries}. Proc. Amer. Math. Soc., {\bf 16} (1965), 1244-1245.
\bibitem{FH} N. Frantzikinakis, B. Host, \emph{The logarithmic Sarnak conjecture for ergodic weights}, Ann. Math. (2) {\bf 187} (3) (2018) 869-931.
\bibitem{GS} A. Granville, K. Soundararajan. \emph{Large character sums: pretentious characters and the Polya-Vinogradov theorem}. J. Amer. Math. Soc., {\bf 20} (2) (2007), 357-384.
\bibitem{HR} H. Helfgott, M. Radziwi\l\l. \emph{Expansion, divisibility and parity}, arXiv:2103.06853.
\bibitem{MR} K. Matomäki, M. Radziwi\l\l. \emph{Multiplicative functions in short intervals}. Ann. of Math. (2), {\bf 183} (3) (2016), 1015-1056.
\bibitem{MR2} K. Matomäki, M. Radziwi\l\l. \emph{Multiplicative functions in short intervals II}, arXiv:2007.04290.
\bibitem{MRT2}  K. Matomäki, M. Radziwi\l\l, T. Tao. \emph{An averaged form of Chowla's conjecture}. Algebra \& Number Theory, {\bf 9} (9) (2015), 2167-2196.
\bibitem{MRT3} K. Matomäki, M. Radziwi\l\l, T. Tao, \emph{Sign patterns of the Liouville and Möbius functions}, Forum Math. Sigma {\bf 4} (2016), e14, 44.
\bibitem{MRT} K. Matomäki, M. Radziwi\l\l, T. Tao. \emph{Fourier uniformity of bounded multiplicative functions in short intervals on average}, Invent. Math., {\bf 220} (1) (2020) 1-58.
\bibitem{MRTTZ} K. Matomäki, M. Radziwi\l\l, T. Tao, J. Teräväinen, T. Ziegler. \emph{Higher uniformity of bounded multiplicative functions in short intervals on average}, Ann. Math. (2) {\bf 197} (2) (2023), 739-857.
\bibitem{Mc} R. McNamara, \emph{Sarnak's conjecture for sequences of almost quadratic word growth}, Ergodic Theory Dynam. Systems {\bf 41} no. 10 (2021), 3060-3115.
\bibitem{Tb} T. Tao. \emph{254A, Notes 2: Complex-analytic multiplicative number theory}, available at https://terrytao.wordpress.com/2014/12/09/254a-notes-2-complex-analytic-multiplicative-number-theory.
\bibitem{T0} T. Tao, \emph{The logarithmically averaged Chowla and Elliott conjectures for two-point correlations}, Forum Math. Pi {\bf 4} (2016), e8, 36.
\bibitem{T} T. Tao. \emph{Equivalence of the logarithmically averaged Chowla and Sarnak conjectures}. In Number theory - Diophantine problems, uniform distribution and applications, 391-421. Springer, 2017.
\bibitem{TT} T. Tao, J. Teräväinen, \emph{The structure of logarithmically averaged correlations of multiplicative functions, with applications to the Chowla and Elliott conjectures}, Duke Math. J. {\bf 168} no. 11 (2019), 1977-2027
\bibitem{Te} J. Teräväinen, \emph{On the Liouville function at polynomial arguments}, arXiv:2010.07924.
\bibitem{W1} M. Walsh, \emph{Local uniformity through larger scales}, Geom. Funct. Anal. {\bf 31} (2021), 981-991.
\bibitem{W2} M. Walsh, \emph{Phase relations and pyramids}, arXiv:2304.09792.
\end{thebibliography}
\end{document}